\documentclass[11pt]{amsart}
\usepackage{pgf,tikz,pgfplots,color}
\pgfplotsset{compat=1.15}
\usepackage{mathrsfs} 
\usetikzlibrary{arrows}
\usepackage{fancyhdr}
\usepackage{lastpage} 

\usepackage[margin=1in]{geometry}
\usepackage{geometry}
\usepackage[toc,page]{appendix}
\usepackage[utf8]{inputenc}
\usepackage{hyperref}
\hypersetup{
	colorlinks=true,
	linkcolor=blue,
	citecolor=brown,
	filecolor=magenta, 
	urlcolor=blue,
}
\usepackage[notcite,notref]{}    
\usepackage{amsmath}
\usepackage{amsfonts}
\usepackage{amssymb}
\usepackage{amsthm}
\usepackage{setspace}
\usepackage{inputenc}
\usepackage[english]{babel} 
\usepackage{comment}
\usepackage[shortlabels]{enumitem}
\numberwithin{equation}{section}

\makeatletter
\def\@tocline#1#2#3#4#5#6#7{\relax
	\ifnum #1>\c@tocdepth 
	\else
	\par \addpenalty\@secpenalty\addvspace{#2}%
	\begingroup \hyphenpenalty\@M
	\@ifempty{#4}{%
		\@tempdima\csname r@tocindent\number#1\endcsname\relax
	}{%
		\@tempdima#4\relax
	}%
	\parindent\z@ \leftskip#3\relax \advance\leftskip\@tempdima\relax
	\rightskip\@pnumwidth plus4em \parfillskip-\@pnumwidth
	#5\leavevmode\hskip-\@tempdima
	\ifcase #1
	\or\or \hskip 1em \or \hskip 2em \else \hskip 3em \fi%
	#6\nobreak\relax
	\hfill\hbox to\@pnumwidth{\@tocpagenum{#7}}\par
	\nobreak
	\endgroup
	\fi}
\makeatother

\title[]{Boundary Regularity of the Bergman Kernel in H\"older space}               
\author[]{Ziming Shi} 

\address{Department of Mathematics,
	Rutgers University - New Brunswick, Piscataway, NJ, 08854}
\email{zs327@rutgers.edu}


\newcommand{\dist}{\operatorname{dist}}

\newcommand{\supp}{\operatorname{supp}}

\newtheorem{thm}{Theorem}[section]
\newtheorem{cor}[thm]{Corollary} 
\newtheorem{prop}[thm]{Proposition}
\newtheorem{lemma}[thm]{Lemma}

\theoremstyle{definition}
\newtheorem{defn}[thm]{Definition}
\newtheorem{exmp}[thm]{Example}
\newtheorem{ques}[thm]{Question}

\theoremstyle{remark}
\newtheorem{rem}[thm]{Remark}
\newtheorem*{clm}{Claim}
\newtheorem*{ack}{Acknowledgment}

\renewcommand{\th}[1]{\begin{thm}\label{#1}}
	\renewcommand{\eth}{\end{thm}}
\newcommand{\co}[1]{\begin{cor}\label{#1}}
	\newcommand{\eco}{\end{cor}}
\newcommand{\pr}[1]{\begin{prop}\label{#1}}
	\newcommand{\epr}{\end{prop}}

\newcommand{\df}[1]{\begin{defn}\label{#1}}
	\newcommand{\edf}{\end{defn}}
\newcommand{\ex}[1]{\begin{exmp}\label{#1}} 
	\newcommand{\eex}{\end{exmp}}
\newcommand{\qu}[1]{\begin{ques}\label{#1}}
	\newcommand{\equ}{\end{ques}}  
\newcommand{\mk}{\begin{rem}}
	\newcommand{\emk}{\end{rem}}
\newcommand{\cl}{\begin{clm}}
	\newcommand{\ecl}{\end{clm}} 
\newcommand{\ac}{\begin{ack}}
	\newcommand{\eac}{\end{ack}} 

\newcommand{\ga}{\begin{gather}}
\newcommand{\ega}{\end{gather}}
\newcommand{\gan}{\begin{gather*}}
\newcommand{\egan}{\end{gather*}}
\newcommand{\al}{\begin{gngn}}
	\newcommand{\eal}{\end{align}}
\newcommand{\aln}{\begin{align*}}
\newcommand{\ealn}{\end{align*}}
\newcommand{\eq}[1]{\begin{equation}\label{#1}}
\newcommand{\eeq}{\end{equation}}

\newcommand{\pa}{\partial{}}
\newcommand{\na}{\nabla}

\newcommand{\we}{\wedge}

\newcommand{\sm}{\setminus}

\newcommand{\DD}[2]{\frac{\partial #1}{\partial #2}}
\newcommand{\pp}[2]{\frac{\partial #1}{\partial #2}}

\newcommand{\R}{\mathbb{R}} 
\newcommand{\C}{\mathbb{C}}

\newcommand{\tit}{\textit} 
\newcommand{\0}{\mathbf{0}}

\newcommand{\ov}{\overline}

\newcommand{\wti}{\widetilde}
\newcommand{\hht}{\widehat}

\newcommand{\RE}{\operatorname{Re}}
\newcommand{\IM}{\operatorname{Im}}
\newcommand{\dbar}{\overline\partial}

\newcommand{\all}{\alpha}

\newcommand{\del}{\delta}
\newcommand{\Del}{\Delta}
\newcommand{\var}{\varphi}

\newcommand{\ve}{\varepsilon}
\newcommand{\om}{\omega}
\newcommand{\Om}{\Omega}

\newcommand{\la}{\lambda}
\newcommand{\ta}{\tau}
\newcommand{\gm}{\gamma}

\newcommand{\si}{\sigma}

\newcommand{\yh}{\frac{1}{2}}

\newcommand{\yf}{\frac{1}{4}}

\newcommand{\re}[1]{(\ref{#1})}

\newcommand{\rl}[1]{Lemma~\ref{#1}}

\newcommand{\rp}[1]{Proposition~\ref{#1}}
\newcommand{\rt}[1]{Theorem~\ref{#1}}

\newcommand{\rrem}[1]{Remark~\ref{#1}}

\newcommand{\nn}{\nonumber}
\newcommand{\nid}{\noindent}

\newcounter{pp}
\newcommand{\bpp}{\begin{list}{$\hspace{-1em}\alph{pp})$}{\usecounter{pp}}}
	\newcommand{\epp}{\end{list}}

\newcounter{ppp}
\newcommand{\bppp}{\begin{list}{$\hspace{-1em}(\roman{ppp})$}{\usecounter{ppp}}}
	\newcommand{\eppp}{\end{list}}

\newcommand{\Kc}{\mathcal{K}} 
\newcommand{\Lc}{\mathcal{L}}

\newcommand{\Pc}{\mathcal{P}}

\newcommand{\Uc}{\mathcal{U}}

\begin{document}
	\definecolor{rvwvcq}{rgb}{0.08235294117647059,0.396078431372549,0.7529411764705882}
	
	\begin{abstract} 
		Let $D$ be a bounded strictly pseudoconvex domain in $\mathbb{C}^n$. Assuming $bD \in C^{k+3+\alpha}$ where $k$ is a non-negative integer and $0 < \alpha \leq 1$, we show that 1) the Bergman kernel $B(\cdot, w_0) \in C^{k+ \min\{\alpha, \frac12 \} } (\overline D)$, for any $w_0 \in D$; 2) The Bergman projection on $D$ is a bounded operator from $C^{k+\beta}(\overline D)$ to $C^{k + \min \{ \all, \frac{\beta}{2} \}}(\overline D) $ for any $0 < \beta \leq 1$. Our results both improve and generalize the work of E. Ligocka. 
	\end{abstract} 
	
	\maketitle
	\tableofcontents 
	\section{Introduction} 
	The main goal of the paper is to prove the following result. 
	\begin{thm} \label{Thm::Bker_intro} 
		Let $D$ be a bounded strictly pseudoconvex domain in $\C^n$ with $C^{k+3+\all}$ boundary, where $k$ is a non-negative integer and $0 < \all \leq 1$. Let $B(z,w)$ be the Bergman kernel for $D$.  Then for every $w_0 \in D$, $B(\cdot, w_0) \in C^{k+ \min\{ \all, \yh \}}(\ov D)$. 
	\end{thm}  
	
Earlier in her paper \cite{Li84}, E. Ligocka showed that if $\Om$ has $C^{k+4}$ boundary for non-negative integers $k$, then $B(\cdot, w_0) \in C^{k+\yh}(\ov D)$. Hence \rt{Thm::Bker_intro} is  an improvement and generalization of Ligocka's result to H\"older spaces. 
	
The study of boundary regularity properties of the Bergman projection and Bergman kernel is of fundamental importance in several complex variables, and the subject has found major applications in the theory of biholomorphic mappings and complex geometry, among many other fields. We mention here some brief history for the results on strictly pseudoconvex domains. When the boundary is $C^\infty$, Kerzman \cite{Ker72} used the theory of $\dbar$-Neumann problem to show that the Bergman kernel function $B(z,w)$ is $C^\infty \times C^\infty(\ov D \times \ov D \sm \Del_{bD})$, where $\Del_{bD}:= \{ (z,w) \in bD \times bD, z=w\}$. 
Soon after, C. Fefferman in his seminal paper \cite{Fef74} gave a description of the behavior of the Bergman kernel $(z,w) \in bD \times bD$ near its singular set $\Del_{bD}$, and as an application he proved the now classical Fefferman's mapping theorem, which states that a biholomorphic mapping $F: D_1 \to D_2$ between two bounded $C^\infty$ strictly pseudoconvex domains $D_1, D_2$ extends to a $C^\infty$ diffeomorphism $\wti F: \ov{D_1} \to \ov{D_2}$. Fefferman's proof was based on the deep properties of the Bergman kernel and Bergman metric on strictly pseudoconvex domains. The analysis however was very difficult and nearly impossible to generalize to other cases. Later on Webster \cite{Web79} and Ligocka-Bell \cite{B-L80} independently found conditions on the boundary behavior of the Bergman kernel that can imply the $C^\infty$ extension of biholomorphic mappings, and consequently they were able to significantly simplify Fefferman's proof.   
	
Phong-Stein \cite{P-S77} and Ahern-Schneider \cite{A-S79} independently proved the H\"older estimates for the Bergman projection. In both work the boundary is assumed to be $C^\infty$ and the proof is based on the work of C.Fefferman \cite{Fef74} and L. Boutet de Monvel and J. Sjöstrand \cite{B-S76}. 
Later on, Ligocka \cite{Li84} constructed a non-orthogonal projection operator with explicit kernels that ``approximates" the Bergman projection operator, and she used it to prove the H\"older estimates assuming boundary is $C^{k+4}$. Ligocka based off her construction on a similar work done by Kerzman-Stein \cite{K-S78} for the Szeg\"o projection on $C^\infty$ strictly pseudoconvex domains. 
The idea is to use the symmetry of the Levi polynomial for the defining function to get a third order cancellation, which then allows one to estimate the singular integrals (see \rp{Prop::cancel}). It is also worthwhile to mention that the method of Kerzman-Stein-Ligocka has been used in a number of subsequent works, for example in \cite{L-S12} and \cite{L-S13}. 
	For a detailed exposition of the work by Ligocka-Bell and Kerzman-Stein-Ligocka, we refer the reader to the book by M.Range \cite[Chapter VII]{Ran86}.   
	
	In this paper we shall give a variant of Ligocka's method which allows us to prove the estimates in H\"older spaces. Our method also has the advantage that the term on the right-hand side of our integral equation behaves much nicer than the one used by Ligocka, which we now explain. 
	Denote the Bergman projection on $D$ by $\Pc$. It is a standard fact that for $w_0 \in D$, one can write $B(\cdot, w_0) = \Pc \var$, where $\var = \var_{w_0} \in C^\infty_c(D)$ (see \rl{Lem::ker_proj}.)
	Ligocka showed that $\Pc \var$ satisfies an integral equation of the form
	\eq{Lig_int_eqn_intro} 
	(I + \Kc) \Pc \var = \Lc^\ast \var.   
	\eeq  
	Here $\Lc$ is a non-orthogonal projection operator mapping $L^2(\Om)$ into $H^2(\Om)$, the $L^2$ Bergman space, $\Lc^\ast$ is the adjoint operator of $\Lc$, and $\Kc:= \Lc^\ast - \Lc$. It was proved in \cite{Li84} that if the boundary is $C^{k+4}$, then $\Kc$ is a compact operator mapping $C^k(\ov D)$ into $C^{k+\yh}(\ov D)$, and $\Lc, \Lc^\ast$ map $C^{k+1}(\ov D)$ (in fact only need derivatives of order $k$ being Lipschitz continuous) into $C^{k+\yh}(\ov D)$. Hence in particular $\Lc^\ast \var \in C^{k+\yh}(\ov D)$. Applying Fredholm theory to the integral equation \re{Lig_int_eqn_intro} then shows that $\Pc \var \in C^{k+\yh}(\ov D)$. 
	
 For our proof we shall use the same operators $\Lc, \Lc^\ast, \Kc$, but instead of considering the integral equation of $\Pc \var$, we show that the following integral equation holds for the function $\Pc \var - \var$ 
	\eq{int_eqn_intro} 
	(I + \Kc) (\Pc \var - \var) = R(\Pc \var - \var), 
	\eeq 
	where $R$ is some operator that maps $P \var - \var$ to a $C^\infty(\ov D)$ function, assuming boundary is only $C^3$. 
	This is in contrast to the right-hand side of \re{Lig_int_eqn_intro}, where the regularity of $\Lc^\ast \var$ depends on the regularity of the boundary and the estimate is much more complicated.   
	
	Using \re{int_eqn_intro}, \rt{Thm::Bker_intro} is then an easy consequence of the following compactness result and Fredholm theory. 
	\begin{prop} \label{Prop::K_bdd_intro}  
		Let $D$ be a bounded strictly pseudoconvex domain in $\C^n$ with $C^{k+3+\all}$ boundary, where $k$ is a non-negative integer and $0 < \all \leq 1$. Then $\Kc$ is a bounded operator from $C^k (\ov D)$ to $C^{k+ \min \{\all, \yh \} }(\ov D)$. 
	\end{prop}
We remark that \rp{Prop::K_bdd_intro} is the main estimate of the paper and takes up the majority of the proof. 
	
Using \rp{Prop::K_bdd_intro} we can also prove the following theorem for the Bergman projection. Similar 
 result has been obtained by Ligocka under the assumption that the boundary is $C^{k+4}$. 
	\begin{thm} \label{Thm::Bproj_intro} 
		Let $D$ be a bounded strictly pseudoconvex domain in $\C^n$ with $C^{k+3+\all}$ boundary, where $k$ is a non-negative integer and $0 < \all \leq 1$. For $0 < \beta \leq 1$, the Bergman projection $\Pc$ for the domain $D$ defines a bounded operator from $C^{k+\beta}(\ov D)$ to $C^{k+ \min \{\all, \frac{\beta}{2} \} } (\ov D)$. 
	\end{thm} 
 In the special case $\all =1$, we recover Ligocka's result. 
Note that \rt{Thm::Bker_intro} can also be obtained as a consequence of \rt{Thm::Bproj_intro}, by the fact that $B =P \var$ and setting $\beta=1$ in \rt{Thm::Bproj_intro}. However we shall give independent proofs of the two theorems based on \rp{Prop::K_bdd_intro}.

The paper is organized as follows. In Section 2, we prove a simple estimate for H\"ormander's $\dbar$ solution operator on pseudoconvex domains. We also prove a refined version of the regularized defining function introduced in \cite{Gong19}, which plays an important role in the proof of \rp{Prop::K_bdd_intro}. In Section 3 we follow Ligocka's idea to construct the operators $\Lc, \Lc^\ast, \Kc$,  using the regularized defining function from Section 2. We then prove various estimates for the kernels of $\Lc, \Lc^\ast, \Kc$. We note that in our proof (\rp{Prop::Kopt_L2_bdd} and the remark after) that $\Lc$ defines a bounded projection operator from $L^2(D)$ to $H^2(D)$, only $C^3$ boundary regularity is needed. 
	
In Section 4 we prove \rp{Prop::K_bdd_intro} and \rt{Thm::Bker_intro}. The proof of \rp{Prop::K_bdd_intro} is splitted into two parts. In the first part, we prove the case for $k=0$, i.e. assuming $bD \in C^{3+\all}$, $0 < \all \leq 1$, we show that $\Kc$ maps  $L^\infty(D)$ boundedly into $C^{\min \{\all,  \yh \} } (\ov D)$. In the second part, we apply the integration by parts techniques from \cite{A-S79} to prove the case for $k \geq 1$.  We next turn to the proof of \rt{Thm::Bker_intro}. First we construct the integral equation \re{int_eqn_intro} using Koppleman's homotopy formula and show that the right-hand side defines a $C^\infty(\ov D)$ function. \rt{Thm::Bker_intro} then follows easily from \rp{Prop::K_bdd_intro} and standard Fredholm theory. 
In Section 5 we prove \rt{Thm::Bproj_intro}. To this end we show that $\Lc$ is a bounded operator from $C^{k+\beta}(\ov D)$ to $C^{k+ \frac{\beta}{2}}(\ov D)$, $0< \beta \leq 1$, assuming boundary is $C^3$.
	
	We now fix some notations used in the paper. The $L^2$ Bergman space on a domain $D$ is denoted by $H^2(D)$. The Bergman projection and Bergman kernel is denoted by $\Pc$ and $B$, respectively.  
We denote by $C^r(\ov D)$ the H\"older space of exponent $r$ on $D$, and $C^\infty_c(D)$ the space of $C^\infty$ functions with compact support in $D$. For simplicity we write $|f|_r := \| f \|_{C^r(\ov D)}$ when the domain $D$ is clear from context.   
	We write $x \lesssim y$ to mean that $x \leq Cy$ for some constant $C$ independent of $x$ and $y$. By $D^l$ we mean a differential operator of order $l$: $D^l_z g(z)= \pa_{z_i}^{\all_i} \pa_{\ov z_j}^{\beta_j} g(z)$, $\sum_i \all_i + \sum_j \beta_j = l$. 

\begin{ack} 
  The author would like to thank the anonymous referee for many valuable suggestions that improve the exposition of the paper. 
\end{ack}
	\section{Preliminaries}  
	\begin{prop} \label{Prop::Hormander}
		Let $D,D'$ be bounded pseudoconvex domains in $\C^n$ such that $D' \subset \subset D$, and let $l \geq 0$. Suppose $\var$ is a $\dbar$-closed $(0,1)$ form in $D$, with coefficients in $W^l(D)$. Let $u = S \var$, where $S$ is H\"ormander's $L^2$ solution operator which solves $\dbar$ on $D$. Then $u \in W^{l+1}(D')$, and 
		\[
		\| u \|_{W^{l+1}(D')} \leq C (\del/2)^{-l-1} \| \var \|_{W^l(D)}, \quad \del:= \dist(D', \pa D),  
		\] 
		where $C$ is an absolute constant depending only on the domain $D$, 
	\end{prop} 
	\begin{proof} 
		By H\"ormander's $L^2$ estimate \cite{Hor65}, we have $\dbar u = \var$ and 
		\eq{L2_est}
		\| u \|_{L^2 (D)} \leq C_0 \| \var \|_{L^2(D)}, 
		\eeq 
		where $C_0$ is a constant which depends only on the the diameter of $D$. 
		Let $\chi \in C^\infty_c(D)$ be such that $\chi \equiv 1$ on $D'$. Further, $\chi$ satisfies the estimate $|D^\gm \chi| \leq \del^{-|\gm|}$, where $\del:= \dist(D',D)$. 
		We use the following fact: If $v \in L^2(\C^n)$ has compact support and $\dbar v \in L^2(\C^n)$, then  
	    \eq{der_L2_equiv} 
		\| \pa_{z_i} v \|_{L^2(\C^n)} =  
		\| \pa_{\ov z_i} v \|_{L^2(\C^n)} . 
		\eeq  
		This can be proved through a simple integration by parts and approximation argument. (see \cite[Lemma 4.2.4]{Hor90}). In what follows we let $D^l$ to denote a differential operator of the form $\prod_{i,j=1}^n \pa_{z_i}^{\all_i} \pa_{\ov z_j}^{\beta_j}$, where $\sum_{i,j=1}^n |\all_i| + |\beta_j| = l$, and we use $\dbar^l$ to denote $\prod_{j=1}^n \pa_{\ov z_j}^{\beta_j}$, where $\sum \beta_j = l$. Applying \re{der_L2_equiv} repeatedly then gives 
		\eq{der_L2_equiv_2}
		\|D^{l+1} v \|_{L^2(\C^n)} 
		= |\dbar^{l+1} v \|_{L^2(\C^n)}
		\eeq 
		for any $v \in L^2(\C^n)$ with compact support and such that $\dbar^{l+1}v \in L^2(\C^n)$. 
		Applying \re{der_L2_equiv_2} with $v = \chi u$, we get 
      \begin{equation} \label{L2_est_prod} 
		\begin{aligned} 
		\| D^{l+1} (\chi u) \|_{L^2(D)} 
		&= \| \dbar^{l+1}(\chi u) \|_{L^2(D)} 
		\\ 
		& \leq \| (\dbar^{l+1} \chi) u \|_{L^2(D)} 
		+ \sum_{1 \leq s \leq l+1} \| (\dbar^{l+1-s} \chi) (\dbar^s u) \|_{L^2(D)}. 
		\end{aligned}
      \end{equation} 
		By \re{L2_est} and estimates for the derivatives of $\chi$, the first integral is bounded by $C_0 \del^{-(l+1)} \| \var \|_{L^2(D)}$. For each integral in the sum, we have for $1 \leq s \leq l+1$ 
		\begin{align*}
		\| (\dbar^{l+1-s} \chi) (\dbar^s u) \|_{L^2(D)}
		&= \| (\dbar^{l+1-s} \chi) (\dbar^{s-1} \dbar u)  \|_{L^2(D)}
		\\ &\leq \del^{-(l+1-s)} \| \var \|_{W^{s-1}(D)}
		\leq \del^{-l} \| \var \|_{W^l(D)}.
		\end{align*} 
		Now, there are in total $  \sum_{k=0}^{l+1} \binom{l+1}{k} = 2^{l+1}$ terms on the right-hand side of \re{L2_est_prod}. Thus by combining the estimates we obtain 
		\[
		\| D^{l+1}(\chi u) \|_{L^2(D)} \leq C_0 2^{l+1} \del^{-(l+1)} \| \var \|_{W^l(D)} = C_0 (\del/2)^{-(l+1)} \| \var \|_{W^l(D)}. 
		\]  
		Since $\chi \equiv 1$ on $D'$, we have 
		\[
		\| D^{l+1} u \|_{L^2(D')} 
		\leq  \| D^{l+1}  (\chi u) \|_{L^2(D)} 
		\leq C_0 (\del/2)^{-(l+1)} \| \var \|_{W^l(D)}. \qedhere
		\] 
	\end{proof}                                     
	
We now show the existence of a defining function that is smooth off the boundary and whose derivatives blow up in a controlled way. 
	
	\begin{prop} \label{Prop::reg_def_fcn} 
Let $D$ be a bounded domain in $\R^N$ with $C^r$ boundary, $r \geq 3$, and let $\rho$ be a defining function of $D$ of the class $C^r$, i.e. there exists a 
$\Uc$ such that $D \subset \subset \Uc$, $\na \rho \neq 0$ on $bD$ and $ D = \{ x \in \Uc: \rho(D) <0 \} $. We denote $|\rho|_r := |\rho|_{C^r(\Uc)}$, where $|\cdot|_{C^r(\Uc)}$ denotes the H\"older-$r$ norm on $\Uc$. Then there exists a defining function $\wti \rho$ of $D$ such that 
		\begin{enumerate}[(a)]
			\item  
			$\wti \rho \in C^r(\R^N) \cap C^\infty( \R^n \sm bD)$. 
			\item
			There exists some $\del_0 > 0$ such that for any $x\notin bD$ and $0< \del(x) := \dist (x, bD) <\del_0$, 
			\[
			|D^j \wti \rho (x)| \lesssim C_j | \rho|_{r}( 1+ \del(x)^{r-j}), \quad j = 0,1,2, \dots, \quad \del(x) := \dist(x,bD).
			\] 
			\item   
			There exists a constant $C$ depending only on the domain $D$ and $|\rho|_3$, and a $\del_1>0$ such that for all $x \in \R^n$ with $\del(x) < \del_1$ the following estimate hold
			\[
			| \hht D^2 \wti \rho (x) - \hht D^2 \rho (x_\ast) | \lesssim C |x - x_\ast|,  \quad |x_\ast-x| := \dist(x, bD). 
			\]
			Here we use $\hht D^2 \rho$ to denote derivatives of $\rho$ of order $2$ and less. 
		\end{enumerate}
	\end{prop}  
	We call $\wti \rho$ a \emph{regularized defining function} of the domain $D$. 
	
	\begin{proof} 
		We will use the argument from \cite{Gong19}. Let $E_r$ be the Whitney extension operator for the domain $D$. By \cite[Lemma 3.7]{Gong19}, $E_r \rho$ is a defining function of $D$ (so that $-E_r \rho$ is a defining function of the domain $(\ov D)^c$, $E_r \rho \in C^r(\R^N) \cap C^\infty ((\ov D)^c)$ and 
		\[
		|D^j E_r \rho(x)| \lesssim C_j |\rho|_r (1+ \del(x)^{r-j}), \quad  j= 0,1,2, \dots. \quad x \in \R^n \sm \ov D . 
		\] 
		Furthermore, for each $x \in \R^n \sm \ov D$ with $0<\del(x)<1$, there exists some constant $C$ depending only on $D$ and $|\rho|_3$ such that $ | \hht{D}^2 (E\rho) (x) - \hht{D}^2 \rho (x_\ast) | \leq C |x -x_\ast| $, where $x_\ast:=  \dist(x, bD)$. 
		Let $E_r'$ be the Whitney extension operator for the domain $(\ov{D})^c$. Then by the same reasoning $\wti \rho := E_r' E_r \rho$ is a defining function of $D$ satisfying $\wti \rho \in C^r(\R^N) \cap C^\infty (\R^n \sm bD)$, and for all $ x \in D$ with $0<\del(x) < \del_1$, the following hold
		\begin{gather*} 
		|D^j \wti \rho (x) | \lesssim C_j' | E_r \rho|_r ( 1+ \del(x)^{r-j}) 
		\lesssim  C_j'' |\rho|_r  ( 1+ \del(x)^{r-j}) , \quad j= 0,1,2,\dots, \quad x \in D;  
		\\  
		| \hht D^2 \wti \rho (x) - \hht D^2 \rho (x_\ast)  |  
		=  | \hht D^2 \wti \rho (x) - \hht D^2 (E_r \rho) (x_\ast)  |  \leq C' |x-x_\ast|, \quad |x_\ast-x| := \dist(x, bD). \qedhere 
		\end{gather*}  
	\end{proof}
 We now state a very useful result to prove H\"older estimates, popularly known as the Hardy-Littlewood lemma. For a proof the reader may refer to \cite[p.~345]{C-S01}.  
\begin{lemma}[Hardy-Littlewood lemma] \label{Lem::H-L} 
		Let $D$ be a bounded domain in $\R^N$ with $C^1$ boundary. Suppose $g \in C^k(\ov D)$ and that for some $0 < \beta < 1$ there is a constant $C$ such that 
		\[
		| D^{k+1} g (x)| \leq C \del(x)^{-1 + \beta}, \quad x \in D, 
		\] 
		where $\del(x) = \dist(x, bD)$. 
		Then $g \in C^{k+\beta}(\ov D)$. 
	\end{lemma}

	The following lemma can be found in \cite{Be93}. We provide the proof for the reader's convenience. 
	\begin{lemma} \label{Lem::ker_proj} 
		Let $D$ be a bounded domain $\Om \subset \C^n$ and let $B(z,w)$ and $\Pc$ denote the Bergman kernel and the Bergman projection for $D$, respectively. Given $w_0 \in D$, there exists a function $\phi_{w_0}$ in $C^\infty_c(D)$ such that
		\begin{equation} \label{ker_proj} 
		\DD{^{|\beta|}}{\ov w^\beta} B(z,w_0) = \Pc \phi^\beta_{w_0}(z), \quad \phi_{w_0}^\beta(z) := (-1)^{|\beta|} \DD{^{|\beta|}}{\ov z^\beta} \phi_{w_0}(z), 
		\end{equation}
		where $\beta$ is a multi-index. 
	\end{lemma}
	\begin{proof}
		Let $\del_0$ denote the distance from $w_0$ to $bD$ and let $B_1(\0)$ the unit ball in $\C^n$. Set
		\[
		\phi_{w_0} (z) = \del_0^{-2n} \phi\left( \frac{z-w_0}{\del_0} \right), \quad z \in D, 
		\] 
		where $\phi$ is a real-valued function in $C^\infty_c(B_1(\0))$ that is radially symmetric about the origin and $\int \phi \, dV = 1$. Clearly, $\phi_{w_0} \in C^\infty_c (D)$. By the property of the Bergman projection and the Bergman kernel, we have
		\begin{align*}
		\Pc \phi_{w_0} (z) &= \int_D B(z, \zeta) \phi_{w_0}(\zeta) \, dV(\zeta)
		\\ &= \int_D B(z,\zeta) \del_0^{-2n} \phi\left( \frac{\zeta-w_0}{\del_0} \right) \, dV(\zeta) 
		\\ &= \int_{B_1(\0)} B(z, \del_0 \zeta + w_0) \phi(\zeta) \, dV(\zeta) 
		\\ &= \ov{ \int_{B_1(\0)} B(\del_0 \zeta + w_0, z ) \phi(\zeta) \, dV(\zeta) }
		\\ &= \ov{B(w_0, z)} = B(z,w_0), 
		\end{align*}
		where we used the fact that $B$ is holomorphic in the first argument and thus both its real and imaginary parts are harmonic functions which satisfy the mean value property. This proves \re{ker_proj} for $\beta=0$. The general case follows similarly by repeating the above calculation and integration by parts. We leave the details to the reader. 
	\end{proof}
	\section{Estimates of the kernel} 
	In this section we follow Ligocka's idea to construct the kernel of the projection operator $\Lc$ for a strictly pseudoconvex domain. For now we assume the defining function $\rho$ is in the class $C^3$.  
	
	Suppose a bounded domain $D \subset \C^n$ is given by $D = \{ z \in \C^n: \rho(z) < 0 \}$. We write
	\[
	D_\del := \{ z \in \C^n: \rho(z) < \del \}, \quad \del>0.   
	\] 
	We shall sometimes write $D_\del(z)$ (or $D_\del(\zeta)$) to indicate that the domain is for the $z$ (or $\zeta$) variable.
	We now construct the kernel to be used in the integral formula. By setting $\rho' = e^{A\rho}-1$, for some large $A$, we see that $\rho'$ is strictly plurisubharmonic in a neighborhood of $\ov D$, and from now on we simply assume $\rho$ satisfies this property. 
	Define 
	\eq{F_def}
	F(z, \zeta) = \sum_{j=1}^n \DD{\rho}{\zeta_j} (\zeta)(\zeta_j - z_j) - \yh \sum_{i,j=1}^n \frac{\pa^2 \rho}{\pa \zeta_i \pa \zeta_j} (\zeta) (z_i - \zeta_i) (z_j - \zeta_j).  
	\eeq
	
	By Taylor's formula we have
	\begin{equation} \label{rho_Taylor} 
	\rho(z) =  \rho(\zeta) - 2 \RE F(z, \zeta)
	+ \Lc_\rho(\zeta; z-\zeta) + o(|z-\zeta|^2), 
	\end{equation} 
	where $\Lc_\rho(\zeta;t)$ is the Levi form of $\rho$ at $\zeta$, i.e. $\Lc_\rho(\zeta;t) := \sum_{i,j=1}^n \frac{\pa^2 \rho }{\pa \zeta_i \pa \ov \zeta_j } t_i \ov{t_j}$.  
	Fix some $\ve_0>0$ small such that for all $z, \zeta \in D_\del$, we have $\Lc_\rho(\zeta; z-\zeta) \geq c|z-\zeta|^2$. It follows from \re{rho_Taylor} that
	\begin{equation} \label{F_lbd}  
	\RE F(z,\zeta)  \geq \frac{\rho(\zeta) - \rho(z)}{2}  + \frac{c}{2} |z-\zeta|^2,  \quad 
	(z, \zeta) \in D_\del \times D_\del, \quad |z-\zeta| < \ve_0;   
	\end{equation} 
	\begin{equation} \label{F-rho_lbd}
	\RE F(z,\zeta) - \rho(\zeta)  \geq -\frac{\rho(\zeta) + \rho(z)}{2}  + \frac{c}{2} |z-\zeta|^2, \quad 
	(z, \zeta) \in D_\del \times D_\del, \quad |z-\zeta| < \ve_0. 
	\end{equation} 
	Let $\chi(t)$ be a smooth cut-off function such that $\chi (t) \equiv 1$ if $t < \frac{\ve_0}{4}$ and $\chi(t) \equiv 0$ if $t > \frac{\ve_0}{2}$. We define 
	the following global support function: 
	\begin{equation} \label{G_def} 
	G(z,\zeta) = \chi ( t ) F(z, \zeta) + ( 1- \chi(t) ) |z -\zeta|^2, \quad t= |z-\zeta|. 
	\end{equation}
 We also define the vector-valued functions $g_0 = (g_0^1, \dots, g_0^n)$ and $g_1 = (g_1^1, \dots, g_1^n)$, where 
\[
  g_0^i (z, \zeta) = \ov{\zeta_i - z_i}, \quad 1 \leq i \leq n, 
\] 
and 	
\eq{g1_def} 
	g_1^i (z, \zeta) = \chi \left( t \right) \left( \sum_{j=1}^n \DD{\rho}{\zeta_j} (\zeta) + \yh \sum_{j=1}^n  \frac{\pa^2 \rho}{\pa \zeta_i \pa \zeta_j} (\zeta) (z_j - \zeta_j) \right) + \left( 1- \chi\left( t \right) \right) ( \ov{\zeta_i - z_i} ), \quad 1 \leq i \leq n, 
	\eeq
	where $t = |z- \zeta|$. It follows that $\left< g_0, \zeta -z \right> = |\zeta -z|^2$ and $\left< g_1, \zeta -z \right> = G(z, \zeta)$. In view of \re{F-rho_lbd} and \re{G_def}, there exists some $c >0$ such that
	\eq{G-rho_lbd_1} 
	\RE G(z,\zeta) - \rho(\zeta) \geq 
	c( -\rho(\zeta) - \rho(z) + |z-\zeta|^2), \quad z, \zeta \in \ov D, 
	\eeq
and 
	\eq{G-rho_lbd_2} 
	\RE G(z,\zeta) - \rho(\zeta) \geq 
	\yf[ -\rho(\zeta) - \rho(z)] + c|z-\zeta|^2, \quad z,\zeta \in D_\del \times D_\del.  
	\eeq 
In particular, \re{G-rho_lbd_1} implies 
\eq{Phi_lbd}
  |G(z,\zeta) - \rho(\zeta)| \gtrsim |\zeta -z|^2, \quad z, \zeta \in D. 
\eeq

	We note that if the boundary is $C^{k+3+\all}$, then $g_1, G \in C^{\infty} \times C^{k+1+\all} (D_\del(z) \times D_\del(\zeta))$  and holomorphic in $z$ whenever $|z-\zeta| < \ve_0/4$.  
	Let
	\[
	\om_\la (z,\zeta) = \frac{1}{2 \pi \sqrt{-1}} \frac{\left<g_\la (z,\zeta), d\zeta \right> }{\left< g_\la, \zeta -z \right>}, \quad \la= 0,1. 
	\] 
	The associated Cauchy-Fantappie forms are given by
	\[
	\Om^\la = \om_\la \we (\dbar_{z,\zeta} \om_\la )^{n-1}, \quad \la = 0,1. 
	\] 
	\[
	\Om^{01} = \om_0 \we \om_1 \we \sum_{k_1 + k_2 = n-2} (\dbar_{z,\zeta} \om_0 )^{k_1} \we  (\dbar_{z,\zeta} \om_1 )^{k_2} . 
	\]  
	We decompose $\Om^\la = \sum_{0 \leq q \leq n} \Om^\la_{0,q}$ and $\Om^{01}= \sum_{0 \leq q \leq n}  \Om^{01}_{0,q}$, where $\Om^\la_{0,q}$ (resp. $\Om^{01}_{0,q}$) has type $(0,q)$ in $z$ and type $(n,n-1-q)$ type in $\zeta$. The following Koppleman's formula holds: 
	\eq{Kop} 
	\dbar_\zeta \Om^{01}_{0,q} + \dbar_z \Om^{01}_{0,q-1} = \Om^0_{0,q} - \Om^1_{0,q},  
	\eeq
	where we take $\Om^{01}_{0,-1} \equiv 0$. Write
	\eq{Om0_def}
	\Om^\la_{0,0} (z,\zeta) = \frac{1}{(2 \pi \sqrt{-1})^n} \frac{1}{\left< g_\la, \zeta-z \right>^n} \left( \sum_{i=1}^n g_\la^i d \zeta_i \right) \we \left( \sum_{i,j =1}^n \dbar_\zeta g_\la^i \we d \zeta_i \right)^{n-1} , \quad \la=0,1;  
	\eeq 
	\eq{Om01_def}
	\Om^{01}_{0,0} (z,\zeta) = \frac{1}{(2 \pi \sqrt{-1})^n} \frac{\left< \ov{\zeta -z}, d\zeta \right>}{|\zeta -z|^2} \we \frac{ \left< g_1, d\zeta \right>}{\left< g_1, \zeta-z \right>} \we 
	\sum_{k_1 + k_2 = n-2}
	\left( \frac{\left< d \ov{\zeta}, d\zeta \right>}{|\zeta -z|^2} \right)^{k_1} \we 
	\left( \frac{ \left< \dbar_\zeta g_1, d\zeta \right>}{\left< g_1, \zeta-z \right>} \right)^{k_2}. 
	\eeq 
	Define
	\begin{equation} \label{N_def} 
	\begin{aligned}
	N(z,\zeta) 
	&:= \frac{1}{(2 \pi \sqrt{-1})^n}  \frac{1}{[G(z,\zeta)-\rho(\zeta)]^n} \left( \sum_{i=1}^n g_1^i(z,\zeta) d \zeta_i \right) \we \left( \sum_{i=1}^n \dbar_\zeta g_1^i(z,\zeta) \we d \zeta_i \right)^{n-1} 
	\\ &= C_n \sum_{i=1}^n (-1)^{i-1} \frac{g_1^i (z,\zeta)}{[G(z,\zeta) - \rho(\zeta)]^n }  
	d\zeta \we (\dbar_\zeta g_1^1 ) \we \cdots \we 
	\hht{\left( \dbar_\zeta g_1^i \right) } \we \cdots \we (\dbar_\zeta g_n), 
	\end{aligned}
	\end{equation}
	where $\hht \eta$ means $\eta$ is being excluded. 
	Note that for $\zeta \in bD$, we have $N(z,\zeta) = \Om^1_{0,0}(z,\zeta)$. Therefore by \re{Kop} with $q =0$, 
	\eq{Kop_N} 
	\Om^0_{0,0}(z,\zeta) = \dbar_\zeta \Om^{01}_{0,0}(z,\zeta) + N(z,\zeta), \quad z \in D_{\ve_0}, \quad \zeta \in bD.
	\eeq  
	
	Let 
	\eq{ker_L_def}   
	L(z, \zeta) dV(\zeta) := \dbar_\zeta N(z,\zeta) - S_z (\dbar_z \dbar_\zeta N)(z,\zeta),   
	\eeq   
	where $S_z$ is H\"ormander's operator that solves $\dbar$ on $D_\del$. In what follows we write
	$L= L_0 + L_1$, where 
\[
 L_0 \, dV(\zeta)= - S_z (\dbar_z \dbar_\zeta N)(z,\zeta), \quad L_1 dV(\zeta) = \dbar_\zeta N(z,\zeta).
\]
For each $\zeta \in \ov D$, $L(\cdot,\zeta)$ is holomorphic on $D$. We also note that if $bD \in C^{k+3+\all}$, then $\dbar_\zeta N \in C^\infty \times C^{k+\all} (D(z) \times D(\zeta))$. In view of \re{G-rho_lbd_2}, \re{N_def} and the fact that $\dbar_z G(z,\zeta), \dbar_z g(z,\zeta) \equiv 0$ for $|z-\zeta| < \ve_0/4$, we see that $\dbar_z \dbar_\zeta N(z,\zeta)$ is a well-defined $\dbar$-closed $(0,1)$ form with coefficients in $C^\infty \times C^{k+\all} (D_\del(z) \times D_\del(\zeta))$, if $\del>0$ is sufficiently small.
	Write
	\begin{equation} \label{L1_exp}
	\begin{aligned} 
	&\dbar_\zeta N (z,\zeta) = L_1(z,\zeta) dV(\zeta)
	\\ &\quad =C_n \frac{\dbar_\zeta [G(z,\zeta) - \rho(\zeta)]}{[G(z,\zeta) - \rho(\zeta)]^{n+1}} \sum_{i=1}^n (-1)^{i-1} g_1^i (z,\zeta) d\zeta \we (\dbar_\zeta g_1^1 ) \we \cdots \we 
	\hht{\left( \dbar_\zeta g_1^i \right) } \we \cdots \we (\dbar_\zeta g_1^n)  
	\\ &\qquad + \frac{1}{[G(z,\zeta) - \rho(\zeta)]^n}
	d \zeta \we (\dbar_\zeta g_1^1) \we \cdots \we (\dbar_\zeta g_1^n) . 
	\end{aligned} 
	\end{equation} 
	In the proof we shall use the following convenient expression from \cite{Li84}: 
	\eq{L1_Lig_exp} 
	L_1(z,\zeta) = \frac{ \eta(\zeta)
		+ O'(|z-\zeta|) 
	}{[G(z,\zeta) - \rho(\zeta)]^{n+1} }, \quad
	\eta(\zeta):= c_n \det 
	\begin{vmatrix}
		\rho(\zeta) & \DD{\rho}{\zeta_i}(\zeta) 
		\\ 
		\DD{\rho}{\ov \zeta_i}(\zeta) & \frac{\pa^2 \rho}{\pa \zeta_i \pa \ov \zeta_j}(\zeta) 
	\end{vmatrix} . 
	\eeq 
	Here we note that $\ov{\eta (\zeta)} = \eta(\zeta)$, and $O'(|z-\zeta|)$ is some linear combination of products of $[ \hht{D^3} \rho (\zeta)] (\zeta_i - z_i)$, where $[\hht{D^3} \rho (\zeta)]$ denotes products of $\rho(\zeta)$ and $D^k_\zeta \rho(\zeta)$, $k \leq 3$. In particular, $O'( |z-\zeta|)$ satisfies the estimates 
 \begin{equation} \label{O'_est} 
   \begin{gathered} 
  | O'( |z-\zeta|)| \lesssim |\rho|_3 |\zeta-z|, \quad |D_z^l O'(|z-\zeta|) |\lesssim |\rho|_3, 
 \\  |D_\zeta^l O'(|z-\zeta|) |\lesssim |\rho|_{l+2} + |\rho|_{l+3} |\zeta -z|, \quad  l \geq 1. 
  \end{gathered}  
 \end{equation}

	We now define the integral operator: 
	\eq{Lc_def}
	\Lc f (z) := \int_D L(z, \zeta) f(\zeta) \,dV(\zeta)  = \int_D \left[ L_0(z,\zeta) + L_1(z,\zeta) \right] f(\zeta) \, dV(\zeta),  
	\eeq
	and the associated adjoint operator
	\[
	\Lc^\ast f (z)
	:= \int_D \ov{L(\zeta, z)} f(\zeta) \,dV(\zeta)
	= \int_D \left[ \ov{L_0(\zeta,z)} + \ov{L_1(\zeta,z)}  \right] f(\zeta) \, dV(\zeta).
	\]  
	In the same way as \re{L1_Lig_exp}, we can also write
	\eq{L_adj_Lig_exp}  
	\ov {L_1(\zeta, z)} = \frac{ \eta(z)
		+ O''(|z-\zeta|) 
	}{[\ov{G(\zeta,z)} - \rho(z) ]^{n+1} }, \quad
	\eta(z):= c_n \det 
	\begin{vmatrix}
		\rho(z) & \DD{\rho}{z_i}(z) 
		\\ 
		\DD{\rho}{\ov z_i}(z) & \frac{\pa^2 \rho}{\pa z_i \pa \ov z_j}(z)
	\end{vmatrix}. 
	\eeq 
	Here $O''(|z-\zeta|) $ is some linear combination of products of $[ \hht{D^3} \rho (z)] (\zeta_i - z_i)$, where $[\hht{D^3} \rho (z)]$ denotes products of $\rho(z)$ and $D^k_z \rho(z)$, $k \leq 3$. $O''(|z-\zeta|) $ satisfies the estimate 
	\begin{equation} \label {O''_est} 
 \begin{gathered} 
		| O''( |z-\zeta|)| \lesssim |\rho|_3 |\zeta-z|, \quad 
		|D_\zeta^l O''(|z-\zeta|) |\lesssim |\rho|_3, \\ 
		|D^l_z O''(|z-\zeta|)| \lesssim  |\rho|_{l+2} + |\rho|_{l+3} |\zeta-z|,  \; l \geq 1. 
  \end{gathered}
\end{equation} 
	Hence if $bD \in C^{k+3+\all}$, then $\ov{L_1(\zeta,z)}$ is $C^{k+\all} \times C^\infty(D(z) \times D(\zeta))$. 
	
	Let
	\begin{equation} \label{Kernel_def} 
	\begin{aligned} 
	K(z, \zeta) &:= \ov{L(\zeta, z)} - L(z, \zeta)
	\\ &= \left[ \ov{L_0(\zeta,z)} - L_0(z,\zeta) \right] + \left[ \ov{L_1(\zeta,z)} - L_1(z,\zeta) \right],    
	\end{aligned}
	\end{equation}
	and 
	\eq{Kc_def} 
	\Kc f(z): = \int_D K(z,\zeta) f(\zeta) \, dV(\zeta) = \int  \left[ \ov{L(\zeta, z)} - L(z, \zeta) \right] f(\zeta) \, dV(\zeta)
	= \Lc^\ast f(z) - \Lc f(z) . 
	\eeq 
	For later purpose we note that $\sqrt{-1} \Kc$ is a self-adjoint operator. 
	
The following cancellation estimate is due to \cite{K-S78}. We include a proof here for the reader's convenience. 
	\begin{prop} \label{Prop::cancel} 
		Let $D$ be a strictly pseudoconvex domain with a $C^3$ defining function $\rho$, with $0 < \all < 1$. Let $F(z,\zeta)$ be the function defined by formula \re{F_def}. Then
		\begin{equation} \label{cancel_est}  
		\left[ F(z,\zeta) - \rho(\zeta) \right] - 
		[ \ov {F(\zeta,z)} - \rho(z) ] = O(|\zeta-z|^3), 
		\end{equation}
		where $|O(|\zeta-z|^3)| \lesssim |\rho|_3 |\zeta-z|^3$. 
	\end{prop}
	\begin{proof} 
		By \re{F_def} we have 
		\begin{align*}
		F(z,\zeta) &= \sum_{j=1}^n \DD{\rho}{\zeta_j} (\zeta)(\zeta_j - z_j) - \yh \sum_{i,j=1}^n \frac{\pa^2 \rho}{\pa \zeta_i \pa \zeta_j}(\zeta) (z_i - \zeta_i) (z_j - \zeta_j)
		\\ &= \sum_{j=1}^n \left[ \DD{\rho}{z_j}(z) + \sum_{k=1}^n \frac{\pa^2 \rho}{\pa z_j \pa z_k} (z)  (\zeta_k - z_k) + \sum_{k=1}^n \frac{\pa^2 \rho}{\pa z_j \pa \ov z_k}(z) (\ov{\zeta_k -z_k}) \right] (\zeta_j -z_j) 
		\\ &\quad - \sum_{i,j=1}^n \yh \frac{\pa^2 \rho}{\pa z_i \pa z_j}(z) (z_i - \zeta_i) (z_j - \zeta_j)
		+ R(z,\zeta)  
		\\ &=\sum_{j=1}^n \DD{\rho}{z_j}(z)(\zeta_j -z_j) + \yh \sum_{i,j=1}^n \frac{\pa^2 \rho}{\pa z_i \pa z_j} (z) (\zeta_i -z_i)(\zeta_j -z_j)
  \\ &\quad + \sum_{i,j=1}^n \frac{\pa^2 \rho}{\pa z_i \pa \ov z_j}(z)(\zeta_i -z_i)(\ov{\zeta_j - z_j}) + R_0(z,\zeta), 
		\end{align*} 
		where we did Taylor expansion for the function $\DD{\rho}{\zeta_j}$ at $z$. Since $\rho \in C^3$, the remainder term $R_0$ satisfies $|R_0(z,\zeta)| \lesssim |\rho|_3 |\zeta-z|^3$. 
		On the other hand, 
		\begin{align*} 
		\ov{F(\zeta,z)} &= \sum_{j=1}^n \DD{\rho}{\ov z_j}(z)(\ov{z_j- \zeta_j}) 
		- \yh \sum_{i,j=1}^n \frac{\pa^2 \rho}{\pa  \ov z_i \pa \ov z_j}(z) \ov{(\zeta_i - z_i)}\ov{(\zeta_j - z_j)}.
		\end{align*} 
		Hence
		\begin{align*}
		&F(z,\zeta) - \ov{F(\zeta, z)}
		= \sum_{j=1}^n \DD{\rho}{z_j} (z)(\zeta_j - z_j) + \sum_{j=1}^n \DD{\rho}{\ov z_j}(z)(\ov{\zeta_j- z_j}) 
		\\ &\quad+ \RE \left( \sum_{i,j=1}^n \frac{\pa^2 \rho}{\pa z_i \pa z_j} (z) (\zeta_i -z_i)(\zeta_j -z_j) \right)
		+ \sum_{i,j=1}^n \frac{\pa^2 \rho}{\pa z_i \pa \ov z_j}(z)(\zeta_i -z_i)(\ov{\zeta_j - z_j}) + R_0(z,\zeta), 
		\end{align*}
		where $|R_0(z,\zeta)| \lesssim |\rho|_3 |\zeta-z|^3$. 
The first four terms on the right-hand side are exactly the first and second order terms in the Taylor polynomial of $\rho$ at $z$, which is equal to $\rho(\zeta) - \rho(z) + R_1(z,\zeta)$, where $|R_1(z,\zeta)| \lesssim |\rho|_3|\zeta -z|^3$. Hence
		\[
		F(z,\zeta) - \ov{F(\zeta, z)} = \rho(\zeta) - \rho(z) + R(z,\zeta), 
		\] 
		where $|R(z,\zeta)| \lesssim |\rho|_3|\zeta-z|^3$. 
	\end{proof}
 In what follows we shall denote
	\[
	\Phi(z,\zeta) = G(z,\zeta) - \rho(\zeta), \quad
	\ov \Phi(\zeta,z) = \ov{G(\zeta,z)} - \rho(z). 
	\] 
	\begin{lemma} \label{Lem::int_est} 
		Let $D$ be a bounded strictly pseudoconvex domain with $C^3$ boundary in $\C^n$, $n\geq 2$, Let $\rho$ be the defining function of $D$. Let $0 < \beta \leq 1$. Let $\Theta(z,\zeta)$ denote either $\Phi(z,\zeta)$ or $\Phi( \zeta,z)$.  
		\begin{enumerate}[(i)] 
			\item Let $0 < \beta \leq 1$. Then
			\begin{equation}  \label{zeta-z_Phi_int_1} 
			\int_D \frac{dV(\zeta)}{|\zeta-z|^{2-\beta} |\Theta(z,\zeta)|^{n+1} }  \lesssim 
   1 + \del(z)^{\frac{\beta}{2} - 1}, 
			\end{equation} 
			 where the constant depends only on $D$.  
			\item
			Let $\beta >0$. Then 
			\begin{equation}  \label{zeta-z_Phi_int_2}  
			\int_{B(z,\tau)} \frac{|z-\zeta|^\beta }{ |\Theta(z,\zeta)|^{n+1} } dV(\zeta)  \lesssim \tau^{\frac{\beta}{2}}, 
			\end{equation}  
			where the constant depends only on $D$.  
		\end{enumerate} 
	\end{lemma} 
	\begin{proof}
		First, we show that for each fixed $z \in D$, there exists a small neighborhood $U_z$ and a coordinate chart $\phi_z: \Uc_z \to \R^{2n}$ with $\phi_z(\zeta) = ((s_1,s_2),t) \in \R^2 \times \R^{2n-2}$ and 
		\eq{Phi_st_lb_est}  
		|\Phi(z,\zeta) |, | \Phi(\zeta, z)| \gtrsim \del(z) + |s_1| + |s_2| + |t|^2, \quad |\zeta-z|  \gtrsim |(s_2,t)|.   
		\eeq  
		Here $\del(z):= \dist (z, bD)$ and in the following computation we shall just write $\del$. 
		We define $s_1 (\zeta) = \rho(\zeta)$ and $s_2(\zeta) = \IM \Phi(z,\zeta) $. Recall that $\Phi(z,\zeta) = F(z,\zeta) - \rho(\zeta)$ when $z, \zeta$ are close and 
		\[
		F(z,\zeta) = \sum_{j=1}^n \DD{\rho}{\zeta_j}(\zeta) ( \zeta_j -z_j) + O ( |\zeta-z|^2). 
		\]
		Hence at $\zeta =z$, we have
		\begin{align*} 
		d_\zeta \IM \Phi(z,\zeta) \we d_\zeta \rho(\zeta) 
		&= d_\zeta \IM F(z,\zeta) \we d_\zeta \rho(\zeta)
		\\ &= \frac{1}{2\sqrt{-1}} ( \pa_\zeta \rho -  \dbar_\zeta \rho) \we (\pa_\zeta \rho + \dbar_\zeta \rho)
		\\ &= \frac{1}{\sqrt{-1}} \pa \rho \we \dbar \rho \neq 0.
		\end{align*}
		We can then find smooth real-valued functions $t_j$, $1 \leq j \leq 2n-2$, with $t_j(\zeta) = 0$ at $\zeta =z$ and 
		\[
		d_\zeta \rho(\zeta) \we d_\zeta \IM \Phi(z,\zeta) \we dt_1(\zeta) \we \cdots \we dt_{2n-2}(\zeta) \neq 0 \quad \text{at $\zeta =z$. }
		\] 
		By the inverse function theorem, $\phi_z =(s_1,s_2,t)$ defines a $C^1$ coordinate map in small neighborhood of $z$. 
		
		To prove the first statement in \re{Phi_st_lb_est}, we use estimate \re{G-rho_lbd_1} which says $ \RE \Phi(z,\zeta) \gtrsim -\rho(\zeta) - \rho(z) + |\zeta-z|^2$, for all $z,\zeta \in D$. It follows that  
		\begin{align*} 
		|\Phi(z,\zeta)| \gtrsim |\RE \Phi (z,\zeta)| + |\IM \Phi(z,\zeta)| 
		\gtrsim \del(z) + |s_1(\zeta)| + |s_2(\zeta)| + |t(\zeta)|^2. 
		\end{align*}
		For $\Phi(\zeta,z)$ the argument goes the same: We note that $ \Phi(\zeta,z) = F(\zeta,z) - \rho(z)$ when $z,\zeta$ are close, and
		\[
		F(\zeta,z) = \sum_{j=1}^n \DD{\rho}{z_j} (z) (z_j -\zeta_j) +  O(|\zeta-z|^2). 
		\]  
		Thus at $\zeta =z$, 
		\begin{align*} 
		\left. d_\zeta \IM \Phi(\zeta,z) \we d_\zeta \rho(\zeta) \right|_{\zeta =z} 
		&= d_\zeta \IM F(\zeta, z) \we d_z \rho(z)
		\\ &= \frac{1}{2 \sqrt{-1} } ( \dbar_z \rho(z) - \pa_z \rho(z) ) \we (\pa_z \rho(z) + \dbar_z \rho(z))
		\\ &= \frac{1}{\sqrt{-1} } \dbar \rho  \we \pa \rho \neq 0.  
		\end{align*}
		The second statement in \re{Phi_st_lb_est} follows from the fact that $s_2(z) = t(z) = 0$. Now, both $\Phi(z,\zeta)$ and $|\zeta-z|$ are bounded below by some positive constant for $\zeta \notin \Uc_z$. Hence using partition of unity in $\zeta$ space, we can bound the integral on the left-hand-side of \re{zeta-z_Phi_int_1} by a constant times
		\[
		\int_0^1  \int_0^1  \int_0^1  \frac{t^{2n-3} ds_1 \, ds_2 \, dt}{(s_2 + t)^{2- \beta} (\del + s_1 + s_2 + t^2)^{n+1} }
		\lesssim \int_0^1 \int_0^1 \frac{ r t^{2n-5+ \beta} \, dr \, dt}{(\del + r + t^2)^{n+1}} := I, 
		\]
		where we used the polar coordinates for $(s_1, s_2)$ with $r = |s|$. We can estimate the integral $I$ by separating into different cases. 
		\\ 
		\textit{Case 1: $\del> r, t^2$. }   
		\[
		I \leq \del^{-(n+1)} \left( \int_0^\del r \, dr \right) \left( \int_0^{\sqrt \del} t^{2n-5+\beta} \, dt \right) \lesssim 1+
		\del^{-n-1+2+\frac{2n-4+ \beta}{2}} = 1 + 
\del^{-1 + \frac{\beta}{2}} . 
		\] 
		\textit{Case 2: $r > \del, t^2$.} 
		\[
		I \leq \int_\del^1 r^{-n} \left( \int_0^{\sqrt r} t^{2n-5+\beta} \, dt \right) \, dr
		\lesssim \int_\del^1 r^{-n+\frac{2n-4+\beta}{2}} \, dr
		\lesssim  1+ \del^{-1+\frac{\beta}{2}}. 
		\] 
		\textit{Case 3: $t^2 > \del,r$.}  
		\[
		I \leq \int_{\sqrt{\del}}^1 \left( \int_0^{t^2} r \, dr \right) t^{2n-5+\beta-2n-2} \, dt
		\lesssim \int_{\sqrt{\del}}^1 t^{\beta -3} \, dt
		\lesssim 1+ \del^{-1+ \frac{\beta}{2}}. 
		\] 
		Combining the estimates we obtain  \re{zeta-z_Phi_int_1}.  
		\\ \\ 
		(ii)  Since $|\Theta(z,\zeta)| \gtrsim |z-\zeta|^2$, the integral is bounded by 
		\begin{align*}
		\int_{B_\tau(z)} \frac{|z-\zeta|^\beta }{ |\Theta(z,\zeta)|^{n+1} } dV(\zeta) 
		&\leq \int_{B_\tau(z)} \frac{dV(\zeta) }{ |z-\zeta|^{2-\beta} |\Theta(z,\zeta)|^n } 
		\\ &\lesssim   \int_0^1 \int_0^1 \int_0^1 \frac{t^{2n-3}\, ds_1 \, d s_2 \, dt}{(s_2+t)^{2-\beta} (\del + s_1+s_2 + t^2)^n } 
		\\ &\lesssim  \int_{s_1=0}^1 \int_{s_2=0}^ \tau \int_{t=0}^\tau \frac{ t^{2n-5+\beta} \, ds_1 \, ds_2 \, dt}{(s_1 + s_2 + t^2)^n}:= I. 
		\end{align*}
		Here we used the fact that $|\zeta-z| \gtrsim (s_2,t)$ and thus $\zeta \in B_\tau(z)$ implies $|s_2|,|t|< \tau$.  
		
		We consider several cases. \\ 
		Case 1: $s_1 > \tau$, the integral is bounded by 
		\begin{align*}
		I \leq \int_\tau^1 \frac{ds_1}{s_1^n} \int_0^\tau \, ds_2 \int_0^\tau t^{2n-5+\beta} \, dt
		\lesssim \tau^{-n+1 +1 + 2n-4+\beta} 
		= \tau^{n-2+\beta} \lesssim \tau^\beta. 
		\end{align*}
		Case 2: $s_1 < \tau$, then we have $|s| < r$, for $s=(s_1,s_2)$. Divide further into subcases. If $t^2>s$, then 
		\begin{align*}
		I \lesssim \int_0^\tau \int_0^\tau \frac{s  t^{2n-5+\beta} \,ds \, dt}{(s+ t^2)^n} 
		\leq \int_0^\tau \left( \int_0^{t^2} s \, ds \right) t^{2n-5+\beta -2n} \, dt
		\lesssim \int_0^\tau t^{\beta-1} \, dt
		\lesssim \tau^\beta. 
		\end{align*}
		On the other hand, if $t^2 < s$, then 
		\begin{align*}
		I &\lesssim \int_0^\tau \left( \int_0^{\sqrt s} t^{2n-3-2+\beta} \, dt \right) \frac{s}{s^n} \, ds \lesssim \int_0^\tau s^{\frac{2n-4+\beta}{2}-n+1}  \, ds \lesssim \int_0^\tau s^{\frac{\beta}{2} -1} \, ds \lesssim \tau^{\frac{\beta}{2}}.  \qedhere
		\end{align*} 
	\end{proof}
	From the proof of \rl{Lem::int_est}, we see that for fixed $\zeta$, we can find a neighborhood $U_\zeta$ of $\zeta$ and a coordinate chart $\phi_\zeta: U_\zeta  \to \R^{2n}$ with  $\phi_\zeta(z) = (s_1', s_2',t') \in \R \times \R \times \R^{2n-2}$.  Indeed, we can set $s_1'(z) = \rho(z)$ and $s_2'(z)= \IM \Phi(z,\zeta)$. At $z=\zeta $, 
	\begin{align*} 
	d_z \IM \Phi(z,\zeta) \we d_z \rho(z) 
	&= d_z \IM F(z,\zeta) \we d_z \rho(z) 
	\\ &= \frac{1}{2 \sqrt{-1}} ( \dbar_\zeta \rho (\zeta) - \pa_\zeta \rho (\zeta)) \we (\pa_\zeta \rho (\zeta)+ \dbar_\zeta \rho (\zeta) )
	\\ &= \frac{1}{\sqrt{-1}} \dbar \rho (\zeta) \we \pa \rho (\zeta) \neq 0.
	\end{align*}
	Hence there exists smooth real-valued functions $t_j'$, $1 \leq j \leq 2n-2$ with $t_j' (z) = 0$ and 
	\[
	d_z \rho(z) \we d_z \IM \Phi(z,\zeta) \we dt_1'(\zeta) \we \cdots \we dt'_{2n-2}(\zeta) \neq 0 \quad \text{at $z = \zeta$. } 
	\] 
	Consequently $(s_1' , s_2', t')$ is  the desired coordinate chart in the $z$ variable.  Now by the same estimate as in the proof of \rl{Lem::int_est}, we can prove the following: 
	\begin{lemma} 
		Keeping the assumptions of \rl{Lem::int_est}.  
		\begin{enumerate}[(i)]
			\item Let $0 < \beta \leq 1$. Then 
			\begin{equation}  \label{zeta-z_Phi_int_3} 
			\int_D \frac{dV(z)}{|\zeta-z|^{2-\beta} |\Theta(z,\zeta)|^{n+1} }  \lesssim 1 + \del(\zeta)^{\frac{\beta}{2} - 1}, \quad \del(\zeta):= \dist(\zeta, bD), 
			\end{equation} 
				where the constant depends only on $D$.  
			\item
			Let $\beta >0$, and denote by $B_\ta (z)$ the ball of radius $\tau$ centered at $z$. Then 
			\begin{equation}  \label{zeta-z_Phi_int_4}   
			\int_{B_\tau(z)} \frac{|z-\zeta|^\beta }{ |\Theta(z,\zeta)|^{n+1} } dV(z)  \lesssim \tau^{\frac{\beta}{2}}, 
			\end{equation} 
			 where the constant depends only on $D$.  
		\end{enumerate} 
	\end{lemma}
  
	\begin{lemma} \label{Lem::Phi_int_est} 
		Let $D$ be a bounded strictly pseudoconvex domain with $C^3$ boundary in $\C^n$, $n\geq 2$, and let $\rho$ be its defining function. Let $\Theta(z,\zeta)$ denote either $\Phi(z,\zeta)$ or $\Phi(\zeta, z)$. Denote $\del(z):= \dist(z, b D)$.  
		\begin{enumerate}[(i)]
			\item 
			\begin{equation} \label{Phi_int_1}   
			\int_D \frac{dV(\zeta)}{ |\Theta(z,\zeta)|^{n+1} }  \lesssim 1+ \log \del(z), \quad z \in D,   
			\end{equation} 
		where the constant depends only on $D$.  
			\item 
			\begin{equation} \label{Phi_int_2}    
			\int_{bD} \frac{d\si (\zeta)}{ |\Theta(z,\zeta)|^n }  \lesssim 1+ \log \del(z), \quad z \in D, 
			\end{equation} 
			where the constant depends only on $D$.   
		\end{enumerate}
	\end{lemma}
	\begin{proof} 
		(i) In the proof we shall write $\del(z)$ simply as $\del$. 
		For fixed $z \in D$, let $\zeta \mapsto (s_1,s_2,t)$ be the coordinate chart in a neighborhood $U_z$ of $z$ as constructed in the proof of \rl{Lem::int_est}. 
		Let $\chi_0$ be a smooth cut-off function such that $\supp \chi_0 \subset E_0(z):= \{ \zeta \in D: -\rho(\zeta) - \rho(z) + |z-\zeta| \leq \si \}$ and $\chi_0 \equiv 1$ on the set $E_1(z): =\{ \zeta \in D: -\rho(\zeta) - \rho(z) + |z-\zeta| \leq \frac{\si}{2} \} $. We choose $\si$ sufficiently small such that $E_0(z) \subset U_z$. 
		Then
	\[ 
		\int_D \frac{dV(\zeta)}{ |\Theta(z,\zeta)|^{n+1} } =  \int_{D \cap E_0} \frac{\chi_0 (\zeta) dV(\zeta)}{ |\Theta(z,\zeta)|^{n+1} } +  \int_{D \sm E_1} \frac{(1-\chi_0(\zeta)) dV(\zeta)}{ |\Theta(z,\zeta)|^{n+1} }. 
	\] 
		In view of \re{G-rho_lbd_1}, the second integral is bounded by a constant independent of $z \in D$. 
		
		The first integral is bounded by 
		\begin{align*}
		\int_{D \cap E_0(z)} \frac{dV(\zeta)}{ |\Theta(z,\zeta)|^{n+1} } &\lesssim \int_0^1 \int_0^1 \int_0^1 \frac{t^{2n-3} \, ds_1 \, ds_2 \, dt}{(\del + s_1+ s_2+ t^2)^{n+1}} 
		\lesssim \int_0^1 \int_0^1 \frac{t^{2n-3} r \, dr \, dt}{(\del + r + t^2)^{n+1}}:= I, 
		\end{align*}
		where we used the polar coordinates for $r =(s_1, s_2)$. 
		We split the integral into the following cases: \\
		\tit{Case 1: $\del+ r \geq t^2$.} \\ 
		\begin{align*}
		I \lesssim \int_0^1 \frac{r}{(\del + r)^{n+1}} \left( \int_0^{\sqrt{\del +r}} t^{2n-3} \, dt \right) \, dr
		\lesssim \int_0^1 (\del + r)^{1-n-1+\frac{2n-2}{2}} \, dr  
		= \int_0^1 (\del +r)^{-1} \, dr \lesssim 1+ \log \del. 
		\end{align*}
		\tit{Case 2: $\del+ r \leq t^2$.} \\ 
		\begin{align*} 
		I \lesssim  \int_0^1 r \left( \int_{\sqrt{\del+r}}^1 \frac{t^{2n-3} \, dt}{t^{2n+2}} \right) \, dr 
		\lesssim \int_0^1 (\del+r)^{-1} \, dr
		\lesssim 1+ \log \del. 
		\end{align*}
		\\
		\medskip 
		(ii) 
		Since $s_1(\zeta) = \rho(\zeta) \equiv 0$ for $\zeta \in bD$, for fixed $z$, there exists some neighborhood $U_z$ of $z$ such that $ \zeta \mapsto (s_2,t)$ is a coordinate chart for $\zeta \in bD \cap U_z$. Let $\chi_0, E_0$ be the same as in the proof of (i). We only have to estimate
		\begin{align*}
		\int_{bD \cap E_0} \frac{ \chi_0 \, d \si(\zeta)}{ |\Theta(z,\zeta)|^n } &\lesssim \int_0^1 \int_0^1 \int_0^1 \frac{t^{2n-3} ds_2 \, dt}{(\del + s_2+ t^2)^n} 
		:= I. 
		\end{align*} 
		Split the integral into two cases. \\ 
		\tit{Case 1: $ \del + s_2 \geq t^2$.} 
		\begin{align*}  
		I \lesssim \int_0^1 \frac{1}{(\del + s_2)^n}  \left( \int_0^{\sqrt{\del+s_2}} t^{2n-3} \, dt  \right) \, ds_2 \lesssim \int_0^1 (\del + s_2)^{-1}  \, ds_2 \lesssim  1 + \log \del. 
		\end{align*}  
		\tit{Case 2: $ \del + s_2 \leq t^2$.} 
		\begin{align*} 
		I \lesssim  \int_0^1 \left( \int_{\sqrt{\del+s_2}}^1 \frac{t^{2n-3} \, dt}{t^{2n}} \right) \, ds_2
		\lesssim \int_0^1 (\del+s_2)^{-1} \, ds_2
		\lesssim 1+ \log \del. 
		\end{align*}  
	\end{proof} 
	We now prove the $L^2$ boundedness of  the operator $\Kc$, assuming boundary is only $C^3$. This result is stated in \cite{Li84} assuming the boundary is $C^4$, and the proof over there uses a much more general estimate from \cite{Kra76}. We shall instead give a direct proof here.  
	\begin{prop} \label{Prop::Kopt_L2_bdd} 
	Let $D$ be a strictly pseudoconvex domain in $\C^n$ with $C^3$ boundary, and let $\Kc$ be the operator given by formula \re{Kc_def}. Then $\Kc$ defines a bounded operator from $L^2 (D)$ to $L^2(D)$. 
	\end{prop}
	\begin{proof} 
		We shall apply Schur's test (see for example \cite{Wol03}), which in our case can be formulated as follows. If 
		\eq{Schur_int_1} 
  \int_D | K(z,\zeta)| \, dV(\zeta) \leq A, \quad \text{for each $z$}, 
		\eeq 
		and 
		\eq{Schur_int_2}
\int_D | K(z,\zeta)| \, dV(z) \leq B, \quad \text{for each $\zeta$}, 	
		\eeq
		then for $f \in L^2 (D)$, $\Kc_f$ defined by the integral $\int_D K(z,\zeta) f(\zeta) \, dV(z)$ converges a.e. and there is an estimate 
		\[
		\| \Kc f \|_{L^2(D)} \leq \sqrt{AB} \| f \|_{L^2(D)}. 
		\]
		Hence it suffices to prove \re{Schur_int_1} and \re{Schur_int_2}. 
		We can write 
		\begin{align*} 
		\int_D |K(z,\zeta)| \, dV(\zeta) &= \int_D \left| \frac{\eta(z) + O'' (|z-\zeta|)}{\ov{\Phi}^{n+1}(\zeta,z) } - \frac{\eta(\zeta) + O'(|z - \zeta|)}{\Phi^{n+1}(z,\zeta) } \right| \, dV(\zeta) 
		\\ &\leq J_1+ J_2 + J_3 + J_4, 
		\end{align*}
		where  we denote 
		\begin{gather*} 
		J_1 =  \int_D \frac{|\eta(z) - \eta(\zeta)|}{|\Phi(\zeta,z)|^{n+1} } \, dV(\zeta),  
		\quad J_2 = \int_D  \eta(\zeta) \left| \frac{1}{\ov \Phi^{n+1} (\zeta,z)} - \frac{1}{\Phi^{n+1} (z,\zeta)} \right| \, dV(\zeta) , 
		\\ 
		J_3 = \int_D \frac{|O''(|z-\zeta|)|}{|\Phi(\zeta,z)|^{n+1} }  \, dV(\zeta), \quad 
		J_4 = \int_D \frac{|O'(|z-\zeta|)|}{|\Phi(z, \zeta)|^{n+1}}  \, dV(\zeta). 
		\end{gather*}
	 By the expression for $\eta$ \re{L1_Lig_exp}, we have $|\eta(z) - \eta(\zeta) | \lesssim |\rho|_3 |\zeta-z|$. We have  
		\begin{align*}  
		J_1 &\lesssim |\rho|_3 \int_D \frac{|\zeta-z|}{|\Phi( \zeta,z)|^{n+1} }  \, dV(\zeta)  
	 \lesssim |\rho|_3, 
		\end{align*}
  where we applied estimate \re{zeta-z_Phi_int_2} in the last inequality. 
		By estimates  \re{Phi_lbd}, \re{cancel_est} and \re{zeta-z_Phi_int_2}, we have 
		\begin{align*}   
		J_2 &\lesssim  |\rho|_2 \int_D \frac{|\Phi^{n+1} (z,\zeta) - \ov{\Phi}^{n+1} (\zeta, z) |}{|\Phi (\zeta, z)|^{n+1} |\Phi^{n+1} (z,\zeta)| } \, dV( \zeta)
		\\ &\lesssim |\rho |_2 \int_D |\Phi(z,\zeta) - \ov{\Phi(\zeta,z)}| \left(   \frac{|\Phi(z,\zeta)|^n + |\Phi(\zeta,z)|^n }{ |\Phi (\zeta, z)|^{n+1} |\Phi^{n+1} (z,\zeta)| }\right) \, dV(\zeta)
		\\ &\lesssim |\rho|_{3} \left( \int_D \frac{|\zeta-z|^{3} }{|\Phi(\zeta, z)|^{n+1} |\Phi(z,\zeta)|} \, dV(\zeta) + \int_D \frac{|\zeta-z|^{3} }{|\Phi(z, \zeta)|^{n+1} |\Phi(\zeta,z)|} \, dV(\zeta) \right) 
		\\ &\lesssim |\rho|_{3} \left(  \int_D \frac{|\zeta-z|}{|\Phi(\zeta,z)|^{n+1}} \, dV(\zeta) + \int_D \frac{|\zeta-z|}{|\Phi(z,\zeta )|^{n+1}} \, dV(\zeta)\right)   
	  \lesssim |\rho|_{3}. 
		\end{align*}
		For $J_3$, we use estimates \re{O'_est}, \re{O''_est} and \re{zeta-z_Phi_int_2}: 
		\begin{align*}   
		J_3 &\lesssim |\rho|_3  \int_D \frac{|\zeta-z|}{|\Phi( \zeta,z)|^{n+1}} \, dV (\zeta)  \lesssim |\rho|_3, 
		\\ 
		J_4 &\lesssim |\rho|_3  \int_D \frac{|\zeta-z|}{|\Phi( z,\zeta)|^{n+1}} \, dV (\zeta)  \lesssim |\rho|_3.   
		\end{align*}  
		Here we note that all the bounds are uniform in $z \in D$.  Hence we have proved \re{Schur_int_2}. 
		In a similar way by using estimate \re{zeta-z_Phi_int_4}, we can prove \re{Schur_int_1}.  The proof is now complete. 
	\end{proof} 
	By using \rp{Prop::Kopt_L2_bdd} and the same argument in \cite{Li84}, we obtain  
	\begin{prop} \label{KP_prop} 
Let $D$ be a strictly pseudoconvex domain in $\C^n$ with $C^3$ boundary, and let $\Lc, \Kc$ be the operators given by formula \re{Lc_def} \re{Kc_def}, respectively. Then the following statements are true. 
		\begin{enumerate}
			\item 
			$\Lc$ is a bounded projection from $L^2(D)$ to $H^2(D)$. In particular, $\Lc$ is the identity map on $H^2(D)$.  
			
			\item 
			$\Pc = \Lc(I-\Kc)^{-1} = (I + \Kc)^{-1} \Lc^\ast $. 
		\end{enumerate}
	\end{prop} 
	It is important to note that unlike the Bergman projection, $\Lc$ is not an orthogonal projection, namely, $\Lc g - g$ is not orthogonal to the Bergman space $H^2(D)$. 
	
	\begin{lemma} \label{Lem::Phi_zeta_der}  
		Let $D$ be a strictly pseudoconvex domain with $C^3$ boundary, and let $D_\del:= \{ z \in \C^n: \rho(z) < 0 \}$. 
		\begin{enumerate}[(i)] 
			\item  
			For all $(z, \zeta) \in D_\del \times D_\del$ with $|z-\zeta|$ sufficiently small, 
			\eq{F-rho_supp_est1}   
			\left| \sum_{i=1}^n \frac{\pa \Phi(z,\zeta)}{\pa \ov \zeta_i } \cdot \DD{\rho}{\zeta_i} \right| > c >0. 
			\eeq 
			\item For each $\zeta_0 \in bD$, there exists a neighborhood $U(\zeta_0)$ and an index $1 \leq j \leq n$ such that $\left| \DD{\rho}{\zeta_j} (\zeta) \right| > c >0$ for all $\zeta \in U(\zeta_0)$. In addition, 
			\eq{F-rho_supp_est2} 
			\DD{\Phi (z,\zeta) }{\zeta_j} \DD{\rho}{ \ov{\zeta_j }}
			- \DD{\Phi(z,\zeta) }{\ov \zeta_j} \DD{\rho}{\zeta_j}  > c' > 0, \quad \forall \:  (z,\zeta) \in U(\zeta_0) \times  U(\zeta_0).   
			\eeq  
			
			\item
			For all $(z, \zeta) \in D_\del \times D_\del$ with $|z-\zeta|$ sufficiently small, 
			\eq{F-rho_supp_est3}  
			\left| \sum_{i=1}^n \frac{\pa \ov {\Phi(\zeta,z)} }{\pa \ov \zeta_i} \cdot \frac{\pa \rho}{\pa \zeta_i} \right| > c > 0. 
			\eeq
			\item
			For each $\zeta_0 \in bD$, there exists a neighborhood $U(\zeta_0)$ and an index $1 \leq j \leq n$ such that $\left| \DD{\rho}{\zeta_j} (\zeta) \right| > c >0$ for all $\zeta \in U( \zeta_0)$. In addition, 
			\eq{F-rho_supp_est4}  
			\DD{\ov {\Phi(\zeta,z)} }{\zeta_j} \DD{\rho}{ \ov \zeta_j}
			- \DD{ \ov{\Phi(\zeta,z)} }{\ov \zeta_j} \DD{\rho}{\zeta_j}  > c' >0, \quad  \forall \:  (z,\zeta) \in U(\zeta_0) \times  U(\zeta_0).   
			\eeq
		\end{enumerate} 
		
	\end{lemma}
	\begin{proof} 
		Compute  
		\begin{equation} \label{Phi_zetabar_der}   
		\begin{aligned}
		\DD{}{\ov \zeta_i} \left[ F(z,\zeta) - \rho(\zeta) \right] 
		&= \DD{}{\ov \zeta_i} \left( \sum_{j=1}^n \DD{\rho}{\zeta_j}(\zeta)(\zeta_j - z_j) - \yh \sum_{j,k=1}^n \frac{\pa^2 \rho}{\pa \zeta_j \pa \ov \zeta_k }(\zeta) (z_j - \zeta_j)(z_k - \zeta_k) - \rho(\zeta)  \right)
		\\ &= -\DD{\rho}{\ov \zeta_i} (\zeta) + O(|\zeta -z|). 
		\end{aligned}
		\end{equation} 
		Estimate \re{F-rho_supp_est1} then follows for $|z-\zeta|$ small since $|\na \rho (\zeta)| > 0$. \\ \\ 
		(ii) Since $d \rho(\zeta_0) \neq 0$, there exists some neighborhood $U(\zeta_0)$ and an index $i_0$ such that $ \left| \DD{\rho}{\zeta_{i_0} } (\zeta) \right| \geq c > 0$ for all $ \zeta \in U(\zeta_0)$. 
		We compute
		\begin{equation} \label{Phi_zeta_der} 
		\begin{aligned} 
		\DD{}{\zeta_{i_0} } \left[ F(z,\zeta) - \rho(\zeta) \right] 
		&= \DD{}{\zeta_{i_0}} \left( \sum_{j=1}^n \DD{\rho}{\zeta_j}(\zeta)(\zeta_j - z_j) -\yh \sum_{j,k=1}^n \frac{\pa^2 \rho}{\pa \zeta_j \pa \zeta_k}(\zeta) (z_j - \zeta_j)(z_k - \zeta_k) - \rho(\zeta) \right) 
		\\ &= \DD{\rho}{\zeta_{i_0}} (\zeta) -\DD{\rho}{\zeta_{i_0}} (\zeta) + O(|\zeta -z|)
		= O(|\zeta -z|). 
		\end{aligned}  
		\end{equation} 
		It follows from \re{Phi_zetabar_der} and \re{Phi_zeta_der} that  
		\begin{align*} 
		\DD{[F(z,\zeta) - \rho(\zeta)]}{\zeta_{i_0}} \DD{\rho}{ \ov \zeta_{i_0}}
		- \DD{[F(z,\zeta) - \rho(\zeta)]}{\ov \zeta_{i_0} } \DD{\rho}{\zeta_{i_0}} 
		= \left| \DD{\rho}{\zeta_{i_0}}(\zeta) \right|^2 + O (|\zeta-z|). 
		\end{align*}  
		Estimate \re{F-rho_supp_est2} then follows if $U(\zeta_0)$ is chosen sufficiently small. 
		\\ 
		(iii) The proof follows similarly by the fact: 
		\begin{equation} \label{Phibar_zetabar_der} 
		\begin{aligned} 
		\DD{}{\ov \zeta_i} \left[ \ov{F(\zeta,z)} - \rho(z) \right] 
		&= \DD{}{\ov \zeta_i} \left( \sum_{j=1}^n \DD{\rho}{\ov z_j}(z) \ov{(z_j - \zeta_j)} - \sum_{j,k=1}^n  \frac{\pa^2 \rho}{\pa \ov z_j \pa \ov z_k}(z) \ov{(\zeta_j - z_j)(\zeta_k - z_k)} - \rho(z) \right) 
		\\ &= - \DD{\rho}{\ov z_i}(z) + O(|\zeta -z|)
		\\ &= - \DD{\rho}{\ov \zeta_i} (\zeta) + O(|\zeta -z|), 
		\end{aligned} 
		\end{equation} 
		where in the last equality we used that $|D \rho(z) - D \rho(\zeta)| \lesssim |\rho|_2 |\zeta-z|$. 
		\\ \\  
		(iv) Since $d \rho(\zeta_0) \neq 0$, there exists some neighborhood $U(\zeta_0)$ and an index $i_0$ such that $ \left| \DD{\rho}{\zeta_{i_0} } (\zeta) \right| \geq c > 0$ for all $ \zeta \in U(\zeta_0)$. Compute
		\begin{align} \label{Phibar_zeta_der}
		\DD{}{\zeta_{i_0}} \left[ \ov{F(\zeta,z)} - \rho(z) \right] 
		&= \DD{}{\zeta_{i_0}} \left( \sum_{j=1}^n \DD{\rho}{\ov z_j}(z) \ov{(z_j - \zeta_j)} - \yh \sum_{j,k=1}^n \frac{\pa^2 \rho}{\pa z_j \pa \ov z_k } (z) \ov{(\zeta_j - z_j)(\zeta_k - z_k)} - \rho(z) \right) 
		=0.
		\end{align}
		It follows from \re{Phibar_zetabar_der} and \re{Phibar_zeta_der} that 
		\begin{align*}
		\DD{ [\ov{F(\zeta,z)} - \rho(z)] }{\zeta_{i_0}} \DD{\rho}{ \ov \zeta_{i_0}}
		- \DD{[\ov{F(\zeta,z)} - \rho(z)] }{\ov \zeta_{i_0}} \DD{\rho}{\zeta_{i_0}} 
		= \left| \DD{\rho}{\zeta_{i_0}}(\zeta) \right|^2 + O (|\zeta-z|).  
		\end{align*} 
		Hence estimate \re{F-rho_supp_est4} holds by choosing $U(\zeta_0)$ sufficiently small.  
	\end{proof} 
	\begin{lemma} \label{Lem::Phi_z_der}  
		Let $D$ be a bounded strictly pseudoconvex domain with a $C^3$ defining function $\rho$, and let $F(z,\zeta)$ be given by \re{F_def}. 
		\begin{enumerate} [(i)]
			\item For each $1 \leq i \leq n$, the following holds for $(z,\zeta) \in D \times D$, 
			\[
			\DD{[F(z,\zeta) - \rho(\zeta)]}{z_i} 
			= - \DD{\rho}{\zeta_i}(\zeta) + O(|\zeta -z|)  ,
			\quad \DD{[F(z,\zeta) - \rho(\zeta)]}{\ov z_i} = O(|\zeta-z|), 
			\] 
			where $|O(|\zeta-z|)| \lesssim |\rho|_2 |\zeta-z| $. 
			\item
			For each $1 \leq i \leq n$, the following holds for $(z,\zeta) \in D_\del \times D_\del$, 
			\[
			\DD{[\ov{F(\zeta,z)} - \rho(z)]}{z_i} 
			= - \DD{\rho}{z_i}(z) + O(|\zeta -z|), \quad 
			\DD{[\ov{F(\zeta,z)} - \rho(z)]}{\ov z_i} = O(|\zeta -z|), 
			\]  
			where $|O(|\zeta-z|)| \lesssim |\rho|_3 |\zeta-z| $. 
		\end{enumerate}
		
	\end{lemma}
	\begin{proof}
		(i) Using definiton of $F$, we have 
		\begin{align*}
		\DD{[F(z,\zeta) - \rho(\zeta)] }{z_i} 
		&= \DD{}{z_i} \left( \sum_{j=1}^n \DD{\rho}{\zeta_j}(\zeta)(\zeta_j - z_j) - \yh \sum_{j,k=1}^n \frac{\pa^2 \rho}{\pa \zeta_j \pa \zeta_k}(\zeta) (z_j - \zeta_j)(z_k - \zeta_k) - \rho(\zeta) \right)
		\\&= - \DD{\rho}{\zeta_i}(\zeta) + O(|\zeta -z|), \quad |O(|\zeta -z|)| \lesssim |\rho|_2 |\zeta-z|, 
		\end{align*}
		and $   \DD{[F(z,\zeta) - \rho(\zeta)]}{\ov z_i} = O(|\zeta-z|)$.  
		\\ 
		(ii) 
		\begin{align*}
		\DD{[\ov{F(\zeta,z)} - \rho(z)]}{z_i} 
		&= \DD{}{z_i} \left( \sum_{j=1}^n \DD{\rho}{\ov z_j}(z) \ov{(z_j - \zeta_j)} - \yh \sum_{j,k=1}^n \frac{\pa^2 \rho}{\pa z_j \pa z_k}(z) \ov{(\zeta_j - z_j)(\zeta_k - z_k)} - \rho(z) \right)
		\\&= - \DD{\rho}{z_i}(z) + O(|\zeta -z|), \quad |O(|\zeta -z|)| \lesssim |\rho|_3 |\zeta-z|, 
		\end{align*}
		and 
		\[ 
		\DD{[\ov{F(\zeta,z)} - \rho(z)] }{\ov z_i} = 
		\DD{\rho}{\ov z_i} - \DD{\rho}{\ov z_i} + O(|\zeta-z|) = O(|\zeta-z|).  \qedhere
		\] 
	\end{proof}
	
	We use the notation: 
	\begin{gather} \label{Q'Q''}
	Q'(z,\zeta) := \sum_{i=1}^n \frac{\pa \Phi(z,\zeta)}{\pa \ov \zeta_i} \DD{\rho}{\zeta_i}, 
	\quad Q''(\zeta,z) := \sum_{i=1}^n \frac{\pa \ov{\Phi(\zeta,z)} }{\pa \ov \zeta_i} \cdot \frac{\pa \rho}{\pa \zeta_i}, 
	\end{gather}
	and we write
	\[
	\left[ d \ov \zeta \right]_i = d \ov{\zeta_1} \we \cdots \we \hht{( d \ov \zeta_i)} \we \cdots 
	\we d \ov{\zeta_n}; 
	\quad 
	\left[ d \zeta \right]_i = d \zeta_1 \we \cdots \we \hht{( d \zeta_i )} \we \cdots 
	\we d \zeta_n. 
	\] 
	\begin{lemma} \label{Lem::Phi_grad_diff}
		For all $(z,\zeta) \in D_\del \times D_\del$ with $|z-\zeta| $ sufficiently small, the following estimates hold
		\begin{enumerate}[(i)] 
			\item 
			\[ 
			|D_z \Phi(z,\zeta) - D_z \ov{\Phi(\zeta, z)}| \lesssim |\rho|_3 |\zeta -z|. 
			\] 
			\item 
			\[ 
			|D_\zeta \Phi(z,\zeta) - D_\zeta \ov{\Phi(\zeta, z)}| \lesssim |\rho|_3 |\zeta -z| . 
			\]
			\item 
			\[
			|Q'(z,\zeta) - Q''(\zeta,z)| \lesssim |\rho|_3 |\zeta-z|. 
			\]
		\end{enumerate}
	\end{lemma}
	\begin{proof}
		This follows immediately from the proof of \rl{Lem::Phi_zeta_der} and \rl{Lem::Phi_z_der}. 
	\end{proof}
	
	We now prove the key integration by parts lemma. This technique was originated by Elgueta \cite{Elg80} and has been developed and used by Ahern and Schneider \cite{A-S79}, Ligocka \cite{Li84}, Lieb-Range \cite{L-R80}, Gong \cite{Gong19}, among others. For our proof we shall mainly follow \cite{A-S79}. We mention that integration by parts is not needed for our results with $C^{3+\all}$ boundary, and that in the subsequent proof the following lemma will only be applied to domains with $C^{k+3+\all}$ boundary, $k \geq 1$.  
	\begin{lemma} \label{Lem::ibp} 
		Let $D$ be a bounded strictly pseudoconvex domain in $\C^n$ with $C^4$ boundary. Suppose $u \in C^1(\ov D)$ and the support of $u$ is contained in some small neighborhood of $z$.  Then the following integration by parts formulae hold
		\begin{enumerate}[(i)] 
			\item 
			\eq{ibp_1}
			\int_D \frac{ u(\zeta) \, dV(\zeta)}{\Phi^{m+1}(z,\zeta)} \, = c_1' \int_{bD} \frac{P'(u)(\zeta) d \si(\zeta)}{\Phi^{m-1}(z,\zeta) }  + c_2'
			\int_D \sum_{i=1}^n \DD{}{\ov \zeta_i}\left( \frac{u (\zeta) \DD{\rho}{\zeta_i}(\zeta) }{Q'(z,\zeta)} \right)  \frac{dV(\zeta)}{\Phi^m(z,\zeta)}. 
			\eeq  
			Here $P'$ is a first order differential operator in $\zeta$ variable (see \re{P'_formula}). 
			\item 
			\eq{ibp_2} 
			\int_D \frac{u(\zeta) dV (\zeta)}{\ov{ \Phi}^{m+1}(\zeta, z) } 
			= c_1'' \int_{b D} \frac{P''(u)(\zeta) \, d\si (\zeta)}{\ov \Phi^{m-1}(\zeta,z) }  
			+ c_2'' \int_D \sum_{i=1}^n \DD{}{\ov \zeta_i}\left( \frac{u (\zeta) \DD{\rho}{\zeta_i}(\zeta) }{Q''(\zeta,z)} \right)  \frac{dV(\zeta)}{\ov{\Phi}^m(\zeta, z)} . 
			\eeq
			Here $P''$ is a first order differential operator in $\zeta$ variable (see \re{P''_formula}). 
			\item 
			\eq{ibp_bdy}   
			\int_{bD} \frac{u(\zeta)  d \si(\zeta)}{\Phi^m(z,\zeta)} =  \int_{bD} \frac{P'(u)(\zeta)  d \si(\zeta)}{\Phi^{m-1}(z,\zeta)}, 
			\quad 
			\int_{bD} \frac{u(\zeta)  d \si(\zeta)}{\Phi^m(\zeta, z)} =  \int_{bD} \frac{P''(u)(\zeta)  d \si(\zeta)}{\Phi^{m-1}(\zeta, z)}. 
			\eeq
			Here $P',P''$ are first order differential operators in $\zeta$. The coefficients of $P'$ (resp.$P''$) involve derivatives of $\rho$ up to order $3$ (resp. order $2$).    
		\end{enumerate}
		
	\end{lemma}
	\begin{proof} 
		In view of \re{G_def} and \re{G-rho_lbd_2}, for each fixed $z \in D$ we have $\Phi(z,\cdot) \in C^1(\ov D)$. By \re{Q'Q''}, and the assumption that $\rho \in C^3$, we see that $Q',Q'' \in C^1(\ov D)$. Hence by Stokes' theorem,
		\begin{equation} \label{ibp_1st_time}
		\begin{aligned} 
		\int_D \frac{ u(\zeta)}{\Phi^{m+1}(z,\zeta)} \, dV(\zeta) 
		&= - \frac{1}{m} \int_{b D} \frac{u(\zeta)}{Q'(z,\zeta) \Phi^m(z,\zeta)} \sum_{k=1}^n (-1)^{k-1} \DD{\rho}{\zeta_k} [d \ov \zeta]_k \we d \zeta 
		\\ &\quad + \frac{1}{m} \int_D \sum_{k=1}^n \DD{}{\ov \zeta_k}\left( \frac{u (\zeta) \DD{\rho}{\zeta_k}(\zeta) }{Q'(z,\zeta)} \right)  \frac{1}{\Phi^m(z,\zeta)} \, dV(\zeta). 
		\end{aligned} 
		\end{equation} 
		To finish the proof we need to apply Stokes' theorem again to the boundary integral. We have on $bD$, 
		\eq{bdy_diff_form} 
		d \rho (\zeta) = \sum_{l=1}^n \left(
		\DD{\rho}{\zeta_l} d \zeta_l 
		+ \DD{\rho}{\ov \zeta_l} d \ov{\zeta}_l \right) \equiv 0. 
		\eeq 
		Let $\{ \chi_\nu \}_{\nu=1}^M$ be a partition of unity of $bD$ subordinate to the cover $\{ U_\nu \}_{\nu=1}^M$. We can assume that on $U_\nu$, there exists an index $i = i(\nu)$ such that $\DD{\rho}{\zeta_{i(\nu)}}(\zeta) \neq 0$. 
		By \re{bdy_diff_form}, we have for $ \zeta \in U_\nu \cap bD$:  
		\begin{align} \label{bdy_diff_form1}   
		d_\zeta \left( \Phi^{-(m-1)} [ d \zeta ]_i \we [ d \ov \zeta ]_i \right) 
		&= - (m-1) \Phi^{-m} \left( \sum_{l=1}^n \DD{\Phi(z,\zeta)}{\zeta_l} d \zeta_l + \DD{\Phi(z,\zeta)}{\ov\zeta_l } d \ov{\zeta}_l \right) \we [ d \zeta ]_i \we [ d \ov \zeta ]_i
		\\ \nn &= -(m-1) \Phi^{-m} \left[ \DD{\Phi(z,\zeta)}{\zeta_i} 
		- \DD{\rho}{\zeta_i} \left( \DD{\rho}{\ov \zeta_i} \right)^{-1} \DD{\Phi(z,\zeta)}{\ov \zeta_i}  
		\right] d \zeta_i \we [d \zeta]_i \we [d \ov \zeta]_i
		\\ \nn &= -(m-1) \Phi^{-m} (-1)^{i -1} a_i(z,\zeta) d \zeta \we [d \ov \zeta]_i, \quad i=i(\nu), 
		\end{align}
		where we set 
		\eq{a_i}
		a_i (z,\zeta) := \DD{\Phi(z,\zeta)}{\zeta_i} 
		- \DD{\rho}{\zeta_i} \left( \DD{\rho}{\ov \zeta_i} \right)^{-1} \DD{\Phi(z,\zeta)}{\ov \zeta_i}, \quad \zeta \in U_\nu \cap bD.  
		\eeq  
		By assumption, $u$ is supported in a small neighorhood of $z$. Hence if for some $\nu$, $\supp u \cap U_\nu$ is non-empty, then $z$ must be sufficiently close to $U_\nu$. Hence in view of estimate \re{F-rho_supp_est2} and by shrinking $U_\nu$ if necessary, we can assume that $a_i(z,\zeta) \geq c >0$ for $\zeta \in \supp u \cap U_\nu$.   
		Accordingly, 
		\[
		\Phi^{-m } d \zeta \we [d \ov \zeta]_{i} 
		= c_m \frac{(-1)^{i-1} }{a_{i}(z,\zeta)}  d_\zeta \left( \Phi^{-(m-1)} [d \zeta]_{i} \we [d \ov \zeta ]_{i} \right), \quad i = i(\nu), \quad \zeta \in U_\nu \cap bD. 
		\]
		Now by \re{bdy_diff_form} we can write 
		\eq{Unu_vol_form} 
		\sum_{k=1}^n (-1)^{k-1} \DD{\rho}{\zeta_k} d \zeta \we [d \ov \zeta]_k = \var_\nu(\zeta) d \zeta \we [d \ov \zeta]_{i(\nu)}, \quad \zeta \in U_\nu \cap bD, 
		\eeq 
		where $\var_\nu$ is a linear combination of products of $\DD{\rho}{\zeta_s}$ and  $\DD{\rho}{\ov{\zeta}_t}$. 
		Hence for the boundary integral in \re{ibp_1st_time} we have 
		\begin{align*}
		&\int_{b D} \frac{u(\zeta) \chi_\nu (\zeta) }{Q'(z,\zeta) \Phi^m(z,\zeta)} \sum_{k=1}^n (-1)^{k-1} \DD{\rho}{\zeta_k} d \zeta \we [d \ov \zeta]_k 
		= \int_{b D} \frac{u(\zeta) \chi_\nu (\zeta)  \var_\nu(\zeta)  }{Q'(z,\zeta)   \Phi^m(z,\zeta)} \, d \zeta \we [d \ov \zeta]_{i(\nu)}
		\\ \quad &= \int_{bD} \frac{u(\zeta) \chi_\nu (\zeta) \var_\nu(\zeta) }{(Q'a_{i(\nu)} )(z,\zeta)} d_\zeta(\Phi^{-(m-1)} [d \zeta]_{i(\nu)} \we [d \ov \zeta]_{i(\nu)} ), 
		\end{align*} 
		where the constant is absorbed into $\var_\nu$. 
		By Stokes' theorem, the integral is equal to 
		\begin{align*}
		\int_{bD} d_\zeta \left(  \frac{u(\zeta) (\chi_\nu \var_\nu)(\zeta)}{(Q'a_{i(\nu)} )(z,\zeta)} \right) \Phi^{-(m-1)} [d \zeta]_{i(\nu)} \we [d \ov \zeta]_{i(\nu)}. 
		\end{align*}  
		Let $\psi_\nu$ be the function such that $ [d \zeta]_{i(\nu)} \we [d \ov \zeta]_{i(\nu)} = \psi_\nu(\zeta) d \si (\zeta) $.  Summing the above expression over $\nu$, the boundary integral in \re{ibp_1st_time} can be written as 
		\eq{P'_formula}
		\int_{bD} P'(u)(\zeta) \Phi^{-(m-1)} \, d\si(\zeta), \quad 
		P'(u)(\zeta):= \sum_{\nu=1}^M d_\zeta \left(  \frac{u(\zeta) (\chi_\nu \var_\nu)(\zeta)}{(Q'a_{i(\nu)} )(z,\zeta)} \right) \psi_\nu(\zeta).  
		\eeq
		Hence we obtain formula \re{ibp_1}. This completes the proof of (i). 
		
		The proof of (ii) goes similar. By Stokes' theorem we have 
		\begin{equation} \label{ibp_1st_time_2} 
		\begin{aligned}  
		\int_D \frac{u(\zeta) }{\ov{ \Phi}^{m+1}(\zeta, z) } \, dV (\zeta)
		&= - \frac{1}{m} \int_{b D} \frac{u(\zeta)}{Q''(\zeta,z) \ov{\Phi}^m(\zeta, z)} \sum_{k=1}^n (-1)^{k-1} \DD{\rho}{\zeta_k} \left[ d \ov \zeta \right]_k  \we d \zeta
		\\ &\quad + \frac{1}{m} \int_D \sum_{k=1}^n \DD{}{\ov \zeta_k}\left( \frac{u (\zeta) \DD{\rho}{\zeta_k}(\zeta) }{Q''(\zeta,z)} \right)  \frac{1}{\ov{\Phi}^m(\zeta, z)} \, dV(\zeta). 
		\end{aligned}  
		\end{equation} 
		Let $\chi_\nu$, $U_\nu$ and $i(\nu)$ be the same as in the proof of (i). By \re{bdy_diff_form}, we have for $\zeta \in U_\nu \cap bD$:  
		\begin{equation} \label{bdy_diff_form2} 
		\begin{aligned} 
		d_\zeta \ov \Phi^{-(m-1)}(\zeta,z) [ d \zeta ]_i \we [ d \ov \zeta ]_i
		&= - m \ov \Phi^{-m}(\zeta,z) \left( \DD{\ov \Phi(\zeta,z)}{\zeta_i} d \zeta_i + \DD{\ov \Phi(\zeta,z)}{\ov\zeta_i } d \ov{\zeta}_i \right) \we [ d \zeta ]_i \we [ d \ov \zeta ]_i
		\\ &= -m \ov \Phi^{-m}(\zeta,z) \left[  \DD{\ov \Phi(\zeta,z)}{\zeta_i} 
		- \DD{\rho}{\zeta_i} \left( \DD{\rho}{\ov \zeta_i} \right)^{-1} \DD{\ov \Phi(\zeta,z)}{\ov \zeta_i}  
		\right] d \zeta_i \we [d \zeta]_i \we [d \ov \zeta]_i
		\\ &= -m \ov \Phi^{-m}(\zeta,z) (-1)^{i -1} b_i(\zeta,z) d \zeta \we [d \ov \zeta]_i, \quad i = i(\nu), 
		\end{aligned}
		\end{equation} 
		where we set
		\eq{b_i}
		b_i (\zeta,z) := \DD{\ov \Phi(\zeta,z)}{\zeta_i} 
		- \DD{\rho}{\zeta_i} \left( \DD{\rho}{\ov \zeta_i} \right)^{-1} \DD{\ov \Phi(\zeta,z)}{\ov \zeta_i}, \quad \zeta \in U_\nu \cap bD. 
		\eeq 
		Using estimate \re{F-rho_supp_est4}, we may assume that  $b_i \geq c > 0$ for $ \zeta \in \supp u \cap U_\nu$. 
		It follows that 
		\[
		\ov \Phi^{-m}(\zeta,z) d \zeta \we [d \ov \zeta]_i = c_m \frac{(-1)^{i-1}}{b_i(\zeta,z)}  d_\zeta \left( \ov \Phi^{-(m-1)}(\zeta,z) [ d \zeta ]_i \we [ d \ov \zeta ]_i \right), \quad i = i(\nu), \quad \zeta \in U_\nu \cap bD. 
		\] 
		By \re{Unu_vol_form}, the boundary integral in \re{ibp_1st_time_2} can be written as 
		\begin{align*}
		&\int_{b D} \frac{u(\zeta) \chi_\nu(\zeta)}{Q''(z,\zeta) \ov{\Phi}^m(\zeta, z)} \sum_{k=1}^n (-1)^{k-1} \DD{\rho}{\zeta_k} \left[ d \ov \zeta \right]_k  \we d \zeta
		= \int_{b D} \frac{u(\zeta) \chi_\nu (\zeta)  \var_\nu(\zeta)  }{Q''(z,\zeta) \ov{\Phi}^m(\zeta,z)} \, d \zeta \we [d \ov \zeta]_{i(\nu)}
		\\ \quad &= \int_{bD} \frac{u(\zeta) \chi_\nu (\zeta) \var_\nu(\zeta) }{(Q''b_{i(\nu)} )(z,\zeta)} d_\zeta(\ov{\Phi}^{-(m-1)}(\zeta,z) [d \zeta]_{i(\nu)} \we [d \ov \zeta]_{i(\nu)} ),  
		\end{align*}
		where the constant is absorbed into $\var_\nu$. 
		By Stokes' theorem, the integral is equal to 
		\begin{align*}
		\int_{bD} d_\zeta \left(  \frac{u(\zeta) (\chi_\nu \var_\nu)(\zeta)}{(Q''b_{i(\nu)} )(z,\zeta)} \right) \ov{\Phi}^{-(m-1)}(\zeta,z) [d \zeta]_{i(\nu)} \we [d \ov \zeta]_{i(\nu)}. 
		\end{align*}  
		Let $\psi_\nu$ be the function such that $ [d \zeta]_{i(\nu)} \we [d \ov \zeta]_{i(\nu)} = \psi_\nu(\zeta) d \si (\zeta) $.  Summing the above expression over $\nu$, the boundary integral in \re{ibp_1st_time_2} can be written as 
		\eq{P''_formula} 
		\int_{bD} P''(u)(\zeta) \ov{\Phi}^{-(m-1)}( \zeta,z) \, d\si(\zeta), \quad 
		P'' (u)(\zeta):= \sum_{\nu=1}^M d_\zeta \left(  \frac{u(\zeta) (\chi_\nu \var_\nu)(\zeta)}{(Q''b_{i(\nu)} )(z,\zeta)} \right) \psi_\nu(\zeta). 
		\eeq
		Hence we obtain formula \re{ibp_2}. 
		
		Finally, the proof of (iii) is clear from the proofs of (i) and (ii). 
	\end{proof}

	In this section we prove \rp{Prop::K_bdd_intro} and then use it to prove \rt{Thm::Bker_intro} and \rt{Thm::Bproj_intro}. First we fix some notations. 
      We will write $|f|_r = |f|_{C^r(\ov D)}$, where $| \cdot |_{C^r(\ov D)}$ denotes the H\"older-$r$ norm on $D$. We also write $\del(z):= \dist(z,bD)$.  
	
\section{Proof of \rp{Prop::K_bdd_intro} and \rt{Thm::Bker_intro}} 
In this section we prove \rt{Thm::Bker_intro}. We begin with \rp{Prop::K_bdd_intro}.
\ 

\medskip
\nid 
\tit{Proof of \rp{Prop::K_bdd_intro}.}  
We shall assume that $\rho$ is a regularized defining function satisfying the properties in \rp{Prop::reg_def_fcn}. In particular, we have $\rho \in C^\infty(\C^n) \cap C^{k+3+\all}(\ov D)$ and 
		\eq{rho_der_est} 
		|D^j \rho(z) | \lesssim C_j |\rho|_{k+2+\all} ( 1+ \del(z)^{k+3+\all-j} ), \quad j=0,1,2,\dots. 
		\eeq
		We recall the notation: 
		\[
		\Phi(z,\zeta):= G(z,\zeta) - \rho(\zeta),
		\quad \ov{\Phi(\zeta,z)} := \ov{G(\zeta,z)} - \rho(z). 
		\] 
		In view of \re{Kernel_def}, we can write
		\[
		\Kc f(z) = \Kc_0 f(z) + \Kc_1 f(z)
		:= \int_D K_0(z,\zeta) f(\zeta) \, dV(\zeta) + \int_D K_1(z,\zeta) f(\zeta) \, dV(\zeta), 
		\] 
		where $K_0(z,\zeta):= \ov{L_0(\zeta, z)} - L_0(z,\zeta)$ and $K_1(z,\zeta) = \ov{L_1(\zeta, z)} - L_1(z,\zeta)$. We first estimate $\Kc_0 f$. In view of \re{Kernel_def}, we have 
		\begin{align*}
		\Kc_0 f(z) &= \int_D f(\zeta) \left( \ov{L_0(\zeta,z)} - L_0(z,\zeta) \right) \, dV(\zeta). 
		\end{align*} 
		where $L_0(z,\zeta) \, dV(\zeta) = S_z (\dbar_z \dbar_\zeta N)(z,\zeta) $. 
		As observed earlier, since $bD \in C^{k+3+\all}$, the coefficients of  $\dbar_z \dbar_\zeta N(z,\zeta)$ belong to the class $C^\infty \times  C^{k+\all}(D_\del(z)  \times  D_\del(\zeta))$. By \rp{Prop::Hormander} and the fact that $S_z$ is a linear operator, we see that $L_0(z,\zeta) \in C^\infty \times  C^{k+\all}(D_\del(z)  \times  D_\del(\zeta))$, which also implies $L_0(\zeta,z) \in C^\infty \times C^{k+\all} (D_\del(\zeta)  \times  D_\del(z))$.  
		
		Accordingly, we have 
		\eq{L0_int_est} 
		\int_D f(\zeta) L_0(z,\zeta) \, dV(\zeta) \in C^\infty(\ov D), \quad  \int_D f(\zeta) \ov{L_0(\zeta,z)} \, dV(\zeta) \in C^{k+\all}(\ov D). 
		\eeq 
		Here the first statement is clear. We now prove the second statement. By estimate \re{rho_der_est} and the expression for $L_0(\zeta,z)$, it follows that 
		\begin{align*}
		\left| \int_D f(\zeta) D^{k+1}_z \ov{L_0(\zeta,z) } \, dV(\zeta) \right| \lesssim |f|_0 (1+ \del(z)^{-1+\all}). 
		\end{align*}
		Hence by \rl{Lem::H-L}, the second integral in \re{L0_int_est} belongs to $C^{k+\all}(\ov D)$. Thus we have shown $\Kc_0 f \in C^{k+\all}(\ov D)$. 
		
		Next we estimate $\Kc_1 f$. First we prove for the case $k=0$, i.e. $\rho \in C^{3+\all}$. 
		In view of \re{L1_Lig_exp}, we have 
		\begin{align*}
		\Kc_1 f(z) 
		&= \int_D f(\zeta) \left( \ov{L_1(\zeta,z)} - L_1(z,\zeta) \right) \, dV(\zeta)
		\\ &= \int_D f(\zeta) \left[ \frac{\eta(z) + O'' (|z-\zeta|)}{\ov{\Phi}^{n+1}(\zeta,z) } - \frac{\eta(\zeta) + O'(|z - \zeta|)}{\Phi^{n+1}(z,\zeta) } \right] \, dV(\zeta). 
		\end{align*} 
		Let $\chi_0$ be a $C^\infty$ cut-off function supported in the set $E_0 :=\{(z,\zeta) \in D \times D: |z-\zeta| < \del_0 \}$, and $\chi_0 \equiv 1$ in $\{(z,\zeta) \in D \times D: |z-\zeta| < \frac{\del_0}{2} \}$, for some $\del_0>0$. From the definition of $G(z,\zeta)$ (see \re{G_def}), we can choose $\del_0$ to be sufficiently small such that on the set $E_0$, we have $\Phi(z,\zeta) = F(z,\zeta) - \rho(\zeta)$.  
		Write
		\[ 
		\Kc_1 f(z) = 
		\Kc_1'f(z) + \Kc_1''f(z), 
		\] 
		where
		\[ 
		\Kc_1'f(z):= \int_D f(\zeta) (\chi_0 K_1)(z,\zeta) \, dV(\zeta), 
		\quad \Kc_1''f(z) = \int_D f(\zeta) [(1-\chi_0) K_1](z,\zeta)  \, dV(\zeta), 
		\] 
  with $K_1 (z,\zeta)= \ov{L_1(\zeta,z)} - L_1(z,\zeta)$. 
		The function $(1-\chi_0) K_1$ is supported in $E_1:= \{ (z,\zeta) \in D \times D: |z-\zeta| \geq \frac{\ve_0}{2} \}$. By estimate \re{G-rho_lbd_1} and the assumption $\rho \in C^{k+3+\all}$, we see that $(1-\chi_0) L_1(z,\zeta) \in C^\infty \times C^{k+\all} (\ov{D}(z) \times \ov{D} (\zeta))$ and $(1-\chi_0) L_1(\zeta,z) \in C^{k+\all} \times C^\infty (\ov D(z) \times \ov D(\zeta))$. By the same argument used to prove \re{L0_int_est}, we can show that $\Kc_1''f \in C^{k+\all}(\ov D)$. 

  It remains to estimate $\Kc_1' f$. We will divide the proof into two steps. In the first part, we show that if $bD \in C^{3+\all}$, then $\Kc_1'f \in C^{\min \{ \all, \yh \}}$. In the second part, we use integration by parts to show that if $bD \in C^{k+3+\all}$, for $k \geq 1$, then $\Kc_1'f \in C^{k+\min \{ \all, \yh \}}$. 
\subsection{Case 1: $bD \in C^{3+\all}$.}   
 \ 
 
Assume now that $bD \in C^{3+\all}$. In what follows we will write $D_0 = D_0(z) = \{ \zeta \in D: |\zeta-z| \leq \del_0 \} $, and without loss of generality we can assume $f$ is supported in $D_0$. 
		Taking $z_j$ derivative we get 
		\begin{align*} 
		&\DD{\Kc_1' f}{z_j}(z) = \int_D f(\zeta) \left( \frac{\DD{}{z_j}[\eta(z) + O''(|z-\zeta|)] }{ \ov{\Phi}^{n+1}(\zeta,z)} - \frac{\DD{}{z_j}[\eta(\zeta) + O'(|z-\zeta|)] }{ \Phi^{n+1}(z,\zeta) } \right) \, dV(\zeta)
		\\ &\quad - (n+1) \int_D f(\zeta) \left( \frac {\left[ \DD{}{z_j} \ov{\Phi(\zeta,z)} \right] (\eta(z) + O''(|z-\zeta|)) }{ \ov{\Phi}^{n+2} (\zeta,z) } 
		- \frac {\left[ \DD{}{z_j} \Phi(z,\zeta) \right] (\eta(\zeta) + O'(|z-\zeta|)) }{ \Phi^{n+2} (z,\zeta) } 
		\right) \, dV(\zeta)
		\\ & := I_1(z) + I_2(z), 
		\end{align*}
		where we denote the first and second integral by $I_1$ and $I_2$, respectively.
		We first estimate $I_1$. By \re{O'_est}, we have 
		\[
		\left| \DD{}{z_j} [\eta(\zeta) + O'(|z-\zeta|)] \right| \lesssim |\rho|_3. 
		\] 
		Hence by \re{Phi_int_1}, 
		\begin{align} \label{I1_int_1}  
		\int_{D} \frac{ \left|\pa_{z_j} [\eta(\zeta) + O'(|z-\zeta|)] \right| }{|\Phi^{n+1}(z,\zeta)|} |f(\zeta)| \, dV(\zeta) 
		\lesssim |f|_0 |\rho|_3 \int_{D} \frac{dV(\zeta)}{|\Phi(z,\zeta)|^{n+1} }  
		\lesssim |f|_0 |\rho|_3  (1 + \log \del(z) ). 
		\end{align}
		On the other hand, using estimate \re{O''_est} we have
		\begin{align*}
		\left| \DD{}{z_j} [\eta(z) + O''(|z-\zeta|)] \right| &\lesssim |D^3_z \rho(z)| + |D^4_z \rho(z)| |\zeta-z| 
		\\ &\lesssim |\rho|_{3+\all} \left(1 + \del(z)^{-1+\all} |\zeta-z| \right), 
		\end{align*} 
		where in the last inequality we applied \re{rho_der_est} with $k=0$ and $j=4$.  
		
		Thus applying \re{zeta-z_Phi_int_2} and \re{Phi_int_1} we obtain`
		\begin{equation} \label{I1_int_2}
		\begin{aligned} 
		&\int_{D} \frac{ \left|\pa_{z_j} [\eta(z) + O''(|z-\zeta|)] \right| }{|\ov{\Phi}^{n+1}(\zeta,z)|} |f(\zeta)| \, dV(\zeta) \\ 
		&\quad \lesssim |\rho|_{3 +\all} |f|_0 \left(  \int_{D} \frac{dV(\zeta)}{|\Phi(\zeta,z)|^{n+1} } + \del(z)^{-1+\all} \int_{D} \frac{|\zeta-z|}{|\Phi(\zeta,z)|^{n+1}} \, dV(\zeta) \right)
		\\ &\quad \lesssim |\rho|_{3 +\all} |f|_0  \left( \log \del(z) + \del(z)^{-1+\all} \right) 
		\lesssim  |\rho|_{3 +\all} |f|_0 \del(z)^{-1+\all}. 
		\end{aligned}  
		\end{equation} 
		Putting together estimates \re{I1_int_1} and \re{I1_int_2}, we get 
		\begin{equation} \label{K1'_I1_est}
		|I_1(z)| \lesssim |\rho|_{3 +\all} |f|_0 \del(z)^{-1+ \all}, \quad 0  < \all < 1. 
		\end{equation}

		For the integral $I_2$, we can write it as $I_2 (z) = -(n+1) \sum_{i=1}^3 J_i(z)$, where 
		\begin{align*} 
		J_1(z) = \int_D f (\zeta) \left( \DD{\ov{ \Phi(\zeta,z)} }{z_j} - \DD{\Phi(z,\zeta)}{z_j}  \right) \frac{[\eta(z) + O''(|z-\zeta|)] }{\ov \Phi^{n+2}(\zeta,z)} \, dV(\zeta) ;  
		\end{align*}
		\begin{align*} 
		J_2(z) = \int_D f (\zeta) \left[ \eta(z) - \eta(\zeta) + O''(|z-\zeta|) - O'(|z-\zeta|)  \right] \frac{\pa_{z_j} \Phi(z,\zeta) }{\ov \Phi^{n+2}(\zeta,z)} \, dV(\zeta) ;  
		\end{align*}
		\begin{align*} 
		J_3(z) = \int_D f (\zeta) 
		\left[ \DD{\Phi(z,\zeta)}{z_j} \right]
		\left[ \eta(\zeta) + O'(|z-\zeta|) \right] \left( \frac{1}{\ov \Phi^{n+2}(\zeta,z)} - \frac{1}{\Phi^{n+2}(z,\zeta)} \right) 
		\, dV(\zeta) .   
		\end{align*}
		To estimate $J_1$, we note that by \rl{Lem::Phi_grad_diff}, for any $\zeta \in \supp f \subset D_0 (z) = \{ |\zeta -z| < \del_0 \}$:  
		\[
		\left| \DD{\ov {\Phi(\zeta,z)}}{z_j} - \DD{\Phi(z,\zeta)}{z_j} \right| \lesssim |\rho|_3 |\zeta -z |. 
		\] 
		Together with estimates \re{O''_est} and \re{Phi_lbd} we get
		\begin{equation} \label{J1_est} 
		|J_1(z)| \lesssim |f|_0 |\rho |_3 \int_D  \frac{|\zeta-z|}{\left| \ov \Phi(\zeta,z) \right|^{n+2} } \, dV(\zeta) 
		\lesssim |f|_0 |\rho |_3 \int_D \frac{dV(\zeta)}{|\zeta-z| |\Phi(\zeta,z)|^{n+1} } 
		\lesssim |f|_0 |\rho |_3 \del(z)^{-\yh}, 
		\end{equation}
		where in the last inequality we applied estimate \re{zeta-z_Phi_int_1} with $\all =1$. 
		For $J_2$, we note that $|\eta (z)- \eta(\zeta)| \lesssim |\rho|_3 |z-\zeta|$. 
		By \rl{Lem::Phi_z_der} (i), we have 
		\eq{Phi_z_der_recall} 
		\left| \DD{\Phi(z,\zeta)}{z_j} \right| 
		\lesssim |\rho|_1 + |\rho|_2 |\zeta-z| \lesssim |\rho|_2. 
		\eeq 
		Applying estimates \re{Phi_lbd}, \re{O'_est}, \re{O''_est}, and \rl{Lem::int_est}, we get
		\begin{equation} \label{J2_est} 
		\begin{aligned} 
		|J_2(z)|  \lesssim |f|_0 |\rho|_3 \int_D   \frac{|\zeta-z| }{\left|  \Phi(\zeta,z) \right|^{n+2}} \, dV(\zeta) 
		\lesssim |f|_0 |\rho |_3 \int_D \frac{dV(\zeta)}{|\zeta-z| |\Phi(\zeta,z)|^{n+1} } 
		\lesssim |f|_0 |\rho |_3 \del(z)^{-\yh}. 
		\end{aligned} 
		\end{equation} 
		For $J_3$ we use estimate \re{cancel_est}, 
		\begin{equation} \label{J3_est} 
		\begin{aligned} 
		|J_3(z)| &\lesssim |f|_0 |\rho|_3
		\int_D |\zeta-z|^3 \left(  \frac{|\Phi(z,\zeta)|^{n+1}}{|\ov {\Phi (\zeta, z)}|^{n+2} |\Phi(z,\zeta)|^{n+2}}
		+ \frac{|\ov{\Phi(\zeta,z)}|^{n+1}}{|\ov {\Phi (\zeta, z)}|^{n+2} |\Phi(z,\zeta)|^{n+2}} \right) \, dV(\zeta) 
		\\ &=|f|_0 |\rho|_3 \int_D |\zeta-z|^3 \left(  \frac{1}{|\Phi (\zeta, z)|^{n+2} |\Phi(z,\zeta)|}
		+ \frac{1}{| \Phi (\zeta, z)| | \Phi(z,\zeta)|^{n+2}} \right) \, dV(\zeta)
		\\ &\lesssim |f|_0 |\rho|_3 
		\left( \int_D \frac{ |\zeta-z| }{| \Phi(\zeta,z)|^{n+2} } \, dV(\zeta) + \int_D \frac{|\zeta-z|}{ | \Phi(z,\zeta)|^{n+2} } \, dV(\zeta) \right)  
		\\ &\lesssim |f|_0 |\rho|_3
		\left( \int_D \frac{dV(\zeta)}{|\zeta-z| |\Phi(\zeta,z)|^{n+1} } + \int_D \frac{dV(\zeta)}{|\zeta-z| |\Phi(z,\zeta)|^{n+1} }  \right)
		\lesssim |f|_0 |\rho|_3  \del(z)^{-\yh}, 
		\end{aligned}
		\end{equation} 
		where in the last inequality we applied \rl{Lem::int_est} with $\beta =1$. 
		Hence we have shown that 
		\eq{K1'_I2_est}
		|I_2(z)| \lesssim |f|_0 |\rho|_3  \del(z)^{-\yh}. 
		\eeq 
		Combining \re{K1'_I1_est} and \re{K1'_I2_est}, we have 
		\[ 
		\left| \DD{\Kc_1' f}{z_j} \right| \lesssim 
		\begin{cases} 
		| f|_0 |\rho|_{3+\all} \del(z)^{-1 +\all }, & \text{if $ 0 < \all \leq \yh$};  \\ 
		| f|_0 |\rho|_3\del(z)^{-\yh}, & \text{if $ \yh \leq \all \leq 1$}.  
		\end{cases}
		\]
		In a similar way we can show that $ |\pa_{\ov z_j}\Kc_1' f |$ satisifes the same estimate. 
		It follows by \rl{Lem::H-L} that 
		\[
		\Kc_1' f \in 	 \begin{cases} 
		C^{\all}(\ov D), & \text{if $ 0 < \all \leq \yh$};  \\ 
		C^{\yh}(\ov D), &\text{if $ \yh \leq \all \leq 1$}.  
		\end{cases}
		\]
	This completes the proof for the $k=0$ case. 
\subsection{Case 2: $bD \in C^{k+3+\all}$, $k \geq 1$} 
 \
 
 We now assume that $\rho \in C^{k+3+\all}$, for $k \geq 1$. Taking $k+1$ derivatives we get 
		\begin{equation} \label{Dgm_Kf}  
		\begin{aligned}
		D^{k+1}_z \Kc_1' f(z) &= \sum_{\gm_1 + \gm_2 \leq k+1}  \int_D f(\zeta) \left[ D^{\gm_1}_z \{ \eta(z) + O''(|z-\zeta|)\} D^{\gm_2}_z (\ov{\Phi}^{-(n+1)}(\zeta,z)) \right] \, dV(\zeta)
		\\ &- \sum_{\gm_1+\gm_2 \leq k+1} \int_D f(\zeta) \left[ D^{\gm_1}_z \{ \eta(\zeta) + O'(|z-\zeta|) \} D^{\gm_2}_z (\Phi^{-(
			n+1)}(z,\zeta)) \right]  \, dV(\zeta) 
		\\ &:= F_1 + F_2, 
		\end{aligned}
		\end{equation}
		where we denote the first and second sum in \re{Dgm_Kf} by $F_1$ and $F_2$, respectively. 
		We break up into cases. \\ 
		\tit{Case 1: $\gm_1 = k+1$. ($\gm_2 =0$).} \\ 
		By \re{O'_est} and \re{O''_est}, we get 
		\[ 
		|D^{k+1}_z \{ \eta(\zeta) + O'|\zeta-z| \} | \lesssim |\rho|_{k+3}.  
		\]
		\[ 
		|D^{k+1}_z \{ \eta(z) + O''|\zeta-z| \} | \lesssim |\rho|_{k+3} + |\rho|_{k+4}   | |\zeta-z| \lesssim |\rho|_{k+3+\all} ( 1+ \del(z)^{-1+\all} |\zeta-z|), 
		\] 
		where for the last inequality we used \re{rho_der_est} with $j = k+4$. 
		By doing similar estimate as that for the integral $I_1$ in the $k=0$ case, we get
		\[ 
		|D^\gm \Kc_1' f (z) | \lesssim |\rho|_{k+3+\all} |f|_0 \del(z)^{-1+ \all}, \quad 0 < \all < 1. 
		\] 
		\tit{Case 2: $1 \leq \gm_2 \leq k$.} ($\gm_1 \leq k$.) \\ 
		The term in the sum in \re{Dgm_Kf} takes the form
		\begin{equation} \label{case2_int} 
		\begin{aligned}
		&\int_D f(\zeta) \left[ \frac{D^{\gm_1}_z \{ \eta(z) + O''(|z-\zeta|)\}}{\ov{\Phi}^{n+1+\tau}(\zeta,z) } S_{\tau}''(z,\zeta)  \right] \, dV(\zeta)
		\\ &\quad - \int_D f(\zeta) \left[ \frac{D^{\gm_1}_z \{ \eta(\zeta) + O'(|z-\zeta|) \}}{\Phi^{n+1+\tau}(z,\zeta)} S_{\tau}'(z,\zeta)   \right]  \, dV(\zeta), \quad \tau \leq \gm_2 \leq k, 
		\end{aligned}  
		\end{equation} 
		where $S''_\tau(z,\zeta)$ is some linear combination of products of $D_z^{l} \ov{\Phi}(\zeta, z)$, $l \leq k$, and $S'_\tau$ is some linear combination of 
		products of $ D_z^l \Phi(z,\zeta)$, $l \leq k$.
		
		It is convenient to recall the notation: 
		\eq{Q'Q''_recall} 
		Q'(z,\zeta) = \sum_{i=1}^n \DD{\Phi(z,\zeta)}{\ov{\zeta}_i} \DD{\rho}{\zeta_i}, \quad Q''(z,\zeta)= \sum_{i=1}^n \frac{\pa \ov{\Phi(\zeta,z)} }{\pa \ov \zeta_i} \cdot \frac{\pa \rho}{\pa \zeta_i}, 
		\eeq 
		and for $|\zeta-z|$ small, we have 
		\begin{gather}
		\label{Phi_recall} 
		\Phi(z,\zeta) = F(z,\zeta) - \rho(\zeta) 
		= \sum_{j=1}^n \DD{\rho}{\zeta_j}(\zeta) (\zeta_j -z_j) - \yh\sum_{i,j=1}^n \frac{\pa^2 \rho}{\pa \zeta_i \pa \zeta_j}(\zeta) (z_i -\zeta_i) (z_j- \zeta_j) - \rho(\zeta) ; 
		\\ 
		\Phi(\zeta,z) = F(\zeta,z) - \rho(z) 
		= \sum_{j=1}^n \DD{\rho}{z_j}(z) (z_j -\zeta_j) - \yh \sum_{i,j=1}^n \frac{\pa^2 \rho}{\pa z_i \pa z_j}(z) (\zeta_i -z_i) (\zeta_j- z_j) - \rho(z). 
		\end{gather}

		For the first integral in \re{case2_int}, we apply integration by parts formulae \re{ibp_1} and \re{ibp_bdy} iteratively until the integral becomes a linear combination of  
		\eq{case2_Phibar_ibp} 
		\int_{bD} \frac{ [D^{\mu_0} f(\zeta)] W_1''(z,\zeta) }{\ov{\Phi}^n(\zeta,z)} \, dV(\zeta), \quad \int_D \frac{[D^{\eta_0} f(\zeta)] W_2''(z,\zeta)}{\ov{\Phi}^{n+1}(\zeta,z)} \, dV(\zeta), \quad \mu_0, \eta_0 \leq k,
		\eeq 
		where $W''_1$, $W''_2$ are some linear combinations of products of 
		\[
		D_\zeta^{\mu_1} D_z^{\gm_1} \{ \eta(z) + O''(|z-\zeta|) \}, 
		\; D^{\mu_2}_\zeta [(Q'')^{-1}], 
		\; D_\zeta^{\mu_3+1} D_z^l \ov{\Phi (\zeta, z)},
		\; D_\zeta^{\mu_4+1} \rho(\zeta), \quad l \leq k,  
		\] 
		and $\mu_i \geq 0$ satisfies $ \sum_{i=1}^4 \mu_i \leq k $. 
		Now we have $|D_\zeta^{\mu_1} D_z^{\gm_1} \{ \eta(z) + O''(|z-\zeta|) \}| \leq C_k |\rho|_{k+2}$ (since $ \gm_1 \leq k$, $\mu_1 \leq k$), $| D^{\mu_2}_\zeta [(Q'')^{-1}(z,\zeta)]| \leq C_k$, $|D_\zeta^{\mu_3+1} D_z^l \ov{\Phi (\zeta, z)}| \leq C_k |\rho|_{k+1} $ (since $l \leq k$). Hence the integrals in \re{case2_Phibar_ibp} and thus the first integral in \re{case2_int} can be bounded by   
		\eq{case2_int_est_1}   
		|f|_k |\rho|_{k+3} \left( \int_{bD} \frac{d\si(\zeta)}{|\Phi (\zeta, z)|^n}  + \int_D \frac{dV(\zeta)}{|\Phi (\zeta, z)|^{n+1}}  \right) 
		\lesssim |f|_k |\rho|_{k+3} (1+\log \del(z)), 
		\eeq  
		where we applied \rl{Lem::Phi_int_est}. 
		For the second integral in \re{case2_int}, we apply formulae \re{ibp_2} and \re{ibp_bdy} iteratively until the integral becomes a linear combination of 
		\eq{case2_Phi_ibp} 
		\int_{bD} \frac{D^{\mu_0} f(\zeta) W_1'(z,\zeta) }{\Phi^n(z, \zeta)} \, dV(\zeta), \quad \int_D \frac{D^{\eta_0} f(\zeta) W_2'(z,\zeta)}{\Phi^{n+1}(z,\zeta)} \, dV(\zeta), \quad \mu_0, \eta_0 \leq k.  
		\eeq  
		Here $W_1'$ and $W_2'$ are linear combinations of products of 
		\[
		D_\zeta^{\mu_1} D^{\gm_1}_z \{ \eta(\zeta) + O'(|z-\zeta|), 
		\; D_\zeta^{\mu_2}[(Q')^{-1}], 
		\; D_\zeta^{\mu_3+1} D_z^l \Phi(z,\zeta), 
		\; D_{\zeta}^{\mu_4+1} \rho(\zeta), \quad l \leq k, 
		\]
		and $\mu_i \geq 0$ satisfies $ \sum_{i=1}^4 \mu_i \leq k $. We have $|D_\zeta^{\mu_1} D^{\gm_1}_z \{ \eta(\zeta) + O'(|z-\zeta|) \}| \leq C_k  |\rho|_{\mu_1+3} \leq |\rho|_{k+3}$ (since $ \gm_1, \mu_1 \leq k$), 
		$ |D_\zeta^{\mu_2}[(Q')^{-1}] | \lesssim |\rho|_{\mu_2+3} \lesssim |\rho|_{k+3}$, and $|D_\zeta^{\mu_3+1} D_z^l \Phi(z,\zeta)| \lesssim C_k |\rho|_{\mu_3+3} \lesssim |\rho|_{k+3}$. It follows that the integrals in \re{case2_Phi_ibp} and hence the second integral in \re{case2_int} is bounded by  
		\eq{case2_int_est_2}
		|f|_k |\rho|_{k+3} \left( \int_{bD} \frac{d\si(\zeta)}{|\Phi (z,\zeta)|^n}  + \int_D \frac{dV(\zeta)}{|\Phi (z, \zeta)|^{n+1}} \right)
		\lesssim |f|_k |\rho|_{k+3} (1+ \log \del(z)).
		\eeq 
		Combining \re{case2_int_est_1} and \re{case2_int_est_2}, we get for this case
		\[
		|D^\gm \Kc_1' f (z) | \lesssim |\rho|_{k+3} |f|_k (1+ \log \del(z)).  
		\]
		\medskip
		\tit{Case 3: $\gm_2 = k+1$ ($\gm_1 =0$).} \\ 
		Applying integration by parts formulae \re{ibp_1}, \re{ibp_2} and \re{ibp_bdy} iteratively to $F_1(z)$ in \re{Dgm_Kf} yields a linear combination of 
		\eq{F1_ibp_ktimes}   
		\int_{bD} \frac{D^{\eta_0} f(\zeta) R''_0(z,\zeta)}{\ov{\Phi}^{n+1}(\zeta,z)}  \, dV(\zeta), \quad 
		\int_D \frac{D^{\mu_0} f(\zeta) R''_1(z,\zeta)}{\ov{\Phi}^{n+2}(\zeta,z)}  \, dV(\zeta),  \quad \eta_0, \mu_0 \leq k. 
		\eeq 
		Similarly we apply integration by parts to $F_2(z)$ until it becomes a linear combination of 
		\eq{F2_ibp_ktimes}   
		\int_{bD} \frac{D^{\eta_0} f(\zeta) R'_0(z,\zeta)}{\Phi^{n+1}(z, \zeta) }  \, dV(\zeta), \quad 
		\int_D \frac{D^{\mu_0} f(\zeta) R'_1(z,\zeta)}{\Phi^{n+2}(z, \zeta) }  \, dV(\zeta). 
		\eeq 
		Here $R''_0(z,\zeta)$ and $R''_1(z,\zeta)$ are some linear combination of products of  
		\[
		D^{\mu_1}_\zeta \left(\eta(z) + O''(|z-\zeta|)  \right), \; D^{\mu_2}_\zeta [(Q'')^{-1}], \;  D^{\mu_3+1}_\zeta \ov{\Phi}(\zeta,z), \; D^{\mu_4}_\zeta D_z \ov{\Phi}(\zeta,z), \; D^{\mu_5+1}_\zeta \rho(\zeta), 
		\] 
		$R'_0(z,\zeta)$ and $R'_1(z,\zeta)$ are some linear combination of the products of
		\eq{mu_i_der} 
		D^{\mu_1}_\zeta \left(\eta(\zeta) + O'(|z-\zeta|)  \right), \; D^{\mu_2}_\zeta [(Q')^{-1}], \;  D^{\mu_3+1}_\zeta \Phi(z,\zeta), \; D^{\mu_4}_\zeta D_z \Phi(z,\zeta), \; D^{\mu_5+1}_\zeta \rho(\zeta),
		\eeq 
		where $0 \leq \mu_i \leq k$ for $0 \leq i \leq 5$, and $\sum_{i=0}^5 \mu_i \leq k$. 
		There are five subcases to consider: 
		\\ \\ 
		\medskip
		\textit{Subcase 1}: $\gm_2 =k+1$, $\mu_0, \mu_1, \mu_2, \mu_3 \leq k-1$. 
		Then we do integration by parts one more time to the integrals in \re{F2_ibp_ktimes} 
		and the resulting integrals become
		\eq{F2_ibp_k+1_times}
		\int_{bD} \frac{D^{ \wti \eta_0} f(\zeta) R'_0(z,\zeta)}{\Phi^n(z, \zeta) }  \, dV(\zeta), \quad 
		\int_{D} \frac{D^{\wti \mu_0} f(\zeta) R'_1(z,\zeta)}{\Phi^{n+1}(z, \zeta) }  \, dV(\zeta), \quad 
		\wti \eta_0, \wti \mu_0 \leq k, 
		\eeq  
		where $\wti R'_0$ and $\wti R'_1$ are linear combinations of products of 
		\[
		D^{\wti{\mu_1}}_\zeta \left(\eta(\zeta) + O'(|z-\zeta|)  \right), \; D^{\wti \mu_2}_\zeta [(Q')^{-1}], \; D^{\wti \mu_3 +1 }_\zeta \Phi(z,\zeta), \; D^{\wti \mu_4}_\zeta D_z \Phi(z,\zeta), \; D^{\wti \mu_5+1}_\zeta \rho(\zeta), 
		\] 
		with $\wti{\mu_0}, \wti{\mu_1}, \wti{\mu_2}, \wti{\mu_3} \leq k$ and $\sum_i \wti \mu_i \leq k+1$. Then 
		\[
		| D^{\wti{\mu_1}}_\zeta \left(\eta(\zeta) + O'(|z-\zeta|)  \right)| \lesssim |\rho|_{k+3}. 
		\] 
		In view of \re{Phi_zetabar_der} and \re{Phi_zeta_der}, we have 
		\eq{Q'_der_est}
		\left| D_\zeta^l Q'(z,\zeta) \right| 
		=  D_\zeta^l 
		\left( \sum_{i=1}^n \DD{ \Phi(z,\zeta)}{\ov{\zeta}_i} \DD{\rho}{\zeta_i} \right) 
		\lesssim |\rho|_{l+2} + | \rho|_{l+3} |\zeta-z| \lesssim |\rho|_{l+3}, 
		\eeq 
		and similarly $		\left| D^{l+1}_\zeta \Phi(z,\zeta) \right| \lesssim |\rho|_{l+3}$. 
		Hence for $\wti \mu_2, \wti \mu_3 \leq k$, we have $ |D^{\wti \mu_2}[(Q')^{-1}], |D^{\wti \mu_3 +1} \Phi(z,\zeta)| \lesssim |\rho|_{k+3}$. Putting together the estimates, it follows that the integrals in \re{F2_ibp_k+1_times} and thus in \re{F2_ibp_ktimes} satisfy
		\begin{gather*}
		\left| \int_{bD} \frac{D^{ \wti \mu_0} f(\zeta) R'_0(z,\zeta)}{\Phi^n(z, \zeta) }  \, dV(\zeta) \right|  
		\lesssim |f|_k |\rho|_{k+3} \int_{bD} \frac{d\si(\zeta)}{|\Phi(z,\zeta)|^n}
		\lesssim |f|_k |\rho|_{k+3} (1+ \log \del(z)); 
		\\ 
		\left|  \int_D \frac{D^{\wti \mu_0} f(\zeta) R'_1(z,\zeta)}{\Phi^{n+1}(z, \zeta) }  \, dV(\zeta) \right| 
		\lesssim |f|_k |\rho|_{k+3} \int_D \frac{dV(\zeta)}{|\Phi(z,\zeta)|^{n+1}} 
		\lesssim |f|_k |\rho|_{k+3} (1+ \log \del(z)). 
		\end{gather*}
		We can obtain similar estimates for the integrals in \re{F1_ibp_ktimes}, where the proof is easier since the functions $R''_0$ and $R''_1$ are $C^\infty$ in $\zeta$.  In conclusion we have shown that in this case
		\[
		| D^\gm_z \Kc_1'f(z)| \lesssim |\rho|_{k+3} |f|_k  (1+ \log \del(z)). 
		\]
		\\ 
		\medskip
		\textit{Subcase 2}: $\gm_2 =k+1$, $\mu_1 = k$. Again we shall only estimate \re{F2_ibp_ktimes} as a similar procedure can be applied to \re{F1_ibp_ktimes}. 
		The integrals in \re{F2_ibp_ktimes} can be written as 
		\eq{mu_1=k_dom_int} 
		\int_{bD} \frac{f(\zeta) D_\zeta^k[\eta(\zeta) + O'|z-\zeta|] R_0'(z,\zeta)}{\Phi^{n+1}(z,\zeta)} \, d\si (\zeta) , \quad
		\int_{D} \frac{f(\zeta) D_\zeta^k[\eta(\zeta) + O'|z-\zeta|] R_1'(z,\zeta)}{\Phi^{n+2}(z,\zeta)}
		\, dV (\zeta), 
		\eeq  
		where $R_0'$ and $R_1'$ are some linear combination of the products of 
		\[
		(Q')^{-1}, \; D_z \Phi(z,\zeta), \; D_\zeta \Phi(z,\zeta), \; D \rho(\zeta). 
		\] 
		We now estimate the domain integral in \re{mu_1=k_dom_int} which can be written as  $B_1+B_2$, where 
		\begin{gather} \label{subcase2_B1B2} 
		B_1 (z)= \int_D \frac{f(\zeta) D_\zeta^k \eta(\zeta) R_1'(z,\zeta) }{\Phi^{n+2}(z,\zeta)} \, dV(\zeta), \quad 
		B_2 (z) = \int_D \frac{f(\zeta) D_\zeta^k [ O'(|z-\zeta|)]  R_1'(z,\zeta)}{\Phi^{n+2}(z,\zeta)}\, dV(\zeta). 
		\end{gather}  
		We apply integration by parts formulae \re{ibp_1} and \re{ibp_bdy} to $B_1$ so that 
		\[ 
		B_1 (z) = \int_{bD} \frac{D^{\wti \nu_0}_\zeta f(\zeta)  D^{k+ \wti{\nu_1}}_\zeta \eta(\zeta)  \wti R'_{10}(z,\zeta)}{\Phi^n (z,\zeta)} \, d\si(\zeta)  +  \int_D \frac{D^{\wti \mu_0}_\zeta f(\zeta)  D^{k+ \wti{\mu_1}}_\zeta \eta(\zeta)  \wti R'_{11}(z,\zeta)}{\Phi^{n+1} (z,\zeta)} \, dV(\zeta). 
		\] 
		Here $ \wti \nu_0, \wti \mu_0, \wti \mu_0,  \wti \mu_1 \leq 1$. $\wti R'_{10}(z,\zeta)$ and $\wti R'_{11}(z,\zeta)$ are linear combinations of the products of 
		\[
		\hht D_\zeta (Q')^{-1}, \quad \hht D_\zeta D_z \Phi(z,\zeta), \quad \hht D_\zeta^2  \Phi(z,\zeta), \quad \hht D_\zeta^2 \rho(\zeta). 
		\]
		In particular $ |\wti R'_{10}(z,\zeta)|, \wti R'_{10}(z,\zeta)  \lesssim |\rho|_4 \lesssim |\rho|_{k+3}$ ($k\geq 1$). It follows from \re{subcase2_B1B2} that
		\begin{align*}
		|B_1(z)| &\lesssim |f|_1 |\rho|_{4} \left( \int_{bD} \frac{d\si(\zeta) }{|\Phi(z,\zeta)|^n} + \int_D \frac{dV(\zeta) }{|\Phi(z,\zeta)|^{n+1}} \right) 
		\\ &\lesssim |f|_1 |\rho|_{k+3} (1+ \log \del(z)), \quad k \geq 1. 
		\end{align*}
		For $B_2$, we use estimate \re{O'_est}: 
		\[
		D_\zeta^k [ O'(|z-\zeta|)] = g_1(z,\zeta) + g_2 (z,\zeta), 
		\] 
		where $|g_1 (z,\zeta)| \lesssim |\rho|_{k+2}$ and $|g_2| \lesssim |\rho|_{k+3}|\zeta-z| $.  
		Write
		\eq{subcase2_B2} 
		B_2(z) = \int_D \frac{f(\zeta) (g_1 R_1')(z,\zeta)}{\Phi^{n+2}(z,\zeta)} \, dV(\zeta) + \int_D \frac{f(\zeta)(g_2 R_1')(z,\zeta)}{\Phi^{n+2}(z,\zeta)} \, dV(\zeta). 
		\eeq
		The second integral is bounded in absolute value by (up to a constant) 
		\[
		|f|_0 |\rho|_{k+3} \int_D \frac{|\zeta-z|}{|\Phi(z,\zeta)|^{n+2} } \, dV(\zeta)  \lesssim |f|_0 |\rho|_{k+3} \int_D 
		\frac{dV(\zeta)}{|\zeta-z| |\Phi(z,\zeta)|^{n+1} } \lesssim \del(z)^{-\yh}. 
		\] 
		For the first integral in \re{subcase2_B2} we apply integration by parts and the resulting integral is bounded up to a constant by 
		\eq{subcase2_ibp}  
		|f|_1 |\rho|_{k+3} \left( \int_{bD} \frac{d \si(\zeta)}{\Phi^{n}(z,\zeta)} 
		+ \int_D \frac{dV(\zeta)}{\Phi^{n+1}(z,\zeta)} \right) \lesssim  |f|_1 |\rho|_{k+3} ( 1+ \log \del(z)). 
		\eeq
		This shows that $|B_2(z)| \lesssim |f|_1 |\rho|_{k+3} \del(z)^{-\yh}$. Combining the estimates we have shown that the domain integral in \re{mu_1=k_dom_int} is bounded by $C |f|_1 |\rho|_{k+3} \del(z)^{-\yh}$. The estimate for the boundary integral in \re{mu_1=k_dom_int} is similar and we leave the details to the reader. In summary we have in this case 
		\[
		| D^\gm_z \Kc_1'f(z)| \lesssim |\rho|_{k+3} |f|_1 \del(z)^{-\yh}. 
		\]
		\\ \medskip 
		\textit{Subcase 3}: $\gm_2 =k+1$, $\mu_2 = k$ in \re{mu_i_der}. From \re{Q'Q''_recall} we can write out $Q'$ as
		\[
		Q'(z,\zeta) = \sum_{i=1}^n \DD{\Phi(z,\zeta)}{\ov{\zeta}_i} \DD{\rho}{\zeta_i}
		= \sum_{i=1}^n \left( - \DD{\rho}{\ov \zeta_i} + O(|\zeta -z|) \right) \DD{\rho}{\zeta_i}, 
		\quad O( |\zeta-z|) \sim \hht D_\zeta^3 \rho (\zeta) (\zeta_i - z_i). 
		\] 
		In view of \re{Phi_zetabar_der} and \re{Phi_zeta_der}, we can write 
		$D^k_\zeta [(Q')^{-1}] = Y_1(z,\zeta) + Y_2(z,\zeta)$, where $|Y_1(z,\zeta)| \lesssim |\rho|_{k+2}$ and $|Y_2(z,\zeta)| \lesssim |\rho|_{k+3} |z-\zeta| $. 
		The integrals in \re{F2_ibp_ktimes} have the form
		\eq{subcase3_ibp}  
		\int_{bD} \frac{f(\zeta) D^k_\zeta[(Q')^{-1}]W_0(z,\zeta)}{ \Phi^{n+1}(z,\zeta)} \, d \si(\zeta),  \quad 
		\int_D \frac{f(\zeta) D^k_\zeta[(Q')^{-1}] W_1(z,\zeta)}{ \Phi^{n+2}(z,\zeta)} \, d V(\zeta). 
		\eeq 
		Here $W_0$ and $W_1$ are some linear combinations of the products of $D_\zeta \rho$, $D_z \Phi(z,\zeta)$, $D_\zeta \Phi (z,\zeta)$ and $\eta(\zeta) +  O'(|z-\zeta|)$. For the domain integral in \re{subcase3_ibp} we have 
		\begin{align}  \label{subcase3_dom_int} 
		\int_D \frac{f(\zeta) D^k_\zeta [(Q')^{-1}] W_1(z,\zeta) }{\Phi^{n+2}(z,\zeta)} \, dV(\zeta)
		= \int_D \frac{f(\zeta) [Y_1 W_1](z,\zeta) }{\Phi^{n+2}(z,\zeta)} \, dV(\zeta) + 
		\int_D \frac{f(\zeta) [Y_2 W_1](z,\zeta) }{\Phi^{n+2}(z,\zeta)} \, dV(\zeta). 
		\end{align}
		For the first term we use integration by parts. 
		Since $|D_\zeta Y_1(z,\zeta) | \lesssim |\rho|_{k+3} $, the resulting integral is bounded by the expression \re{subcase2_ibp}. 
		For the second term in \re{subcase3_dom_int} we estimate directly: 
		\begin{align*}
		\left| \int_D \frac{f(\zeta) [Y_2 W_1](z,\zeta) }{\Phi^{n+2}(z,\zeta)} \, dV(\zeta) \right|  &\lesssim |f|_0 |\rho|_{k+3} \int_D \frac{|\zeta-z|}{|\Phi(z,\zeta)|^{n+2}} \, dV(\zeta) 
		\\ &\lesssim |f|_0 |\rho|_{k+3} \int_D \frac{dV(\zeta)}{|\zeta-z| |\Phi(z,\zeta)|^{n+1} } \lesssim  |f|_0 |\rho|_{k+3} \del(z)^{-\yh}, 
		\end{align*}
		where in the last inequality we applied estimate \re{zeta-z_Phi_int_1} with $\beta = 1$.  Thus the absolute value of the domain integral in \re{subcase3_ibp} is bounded up to constant by $|\rho|_{k+3} |f|_1  \del(z)^{-\yh}$.  We can similarly show the same bound for the boundary integral in \re{subcase3_ibp}. Hence in this case  
		\[
		| D^\gm_z \Kc_1'f(z)| \lesssim |\rho|_{k+3} |f|_1 \del(z)^{-\yh}. 
		\]
		\\ \medskip
		\textit{Subcase 4}: $\gm_2 =k+1$, $\mu_3 = k$ in \re{mu_i_der}. 
		Then the integrals in \re{F2_ibp_ktimes} take the form 
		\[
		\int_{bD} \frac{f(\zeta) D^{k+1}_\zeta \Phi(z,\zeta) W_0(z,\zeta)}{\Phi^{n+1}(z,\zeta)} \, dV(\zeta)
		, \qquad \int_D \frac{f(\zeta) D^{k+1}_\zeta \Phi(z,\zeta)  W_1(z,\zeta)}{\Phi^{n+2}(z,\zeta)} \, dV(\zeta), 
		\] 
		where $W_0, W_1$ are some linear combinations of the products of $D_\zeta \rho( \zeta)$, $D_z \Phi(z,\zeta)$, $D_\zeta \Phi(z,\zeta)$ and $\eta(\zeta) + O'(|z-\zeta|)$. 
		As in the subcase 3 we can write 
		$D^{k+1}_\zeta \Phi(z,\zeta) = Y_1 + Y_2 $, where $|D^{k+1}_\zeta Y_1(z,\zeta)| \lesssim |\rho|_{k+2}$ and $|D^{k+1}_\zeta Y_2(z,\zeta)|  \lesssim |\rho|_{k+3} |\zeta -z| $. The rest of the estimates are the same as in Subcase 3. 
		\\ \\ 
		\textit{Subcase 5: $\gm_2 =k+1$, $\mu_0 = k$}. 
		Then the integrals in \re{F1_ibp_ktimes} can be written as 
		\[
		\int_{bD} \frac{D^k f(\zeta) A_0''(z,\zeta) }{\ov{\Phi}^{n+1}(\zeta,z)} \, d\si(\zeta), \quad
		\int_D \frac{D^k f(\zeta) A_1''(z,\zeta) }{\ov{\Phi}^{n+2}(\zeta,z)} \, dV(\zeta), 
		\]
		where $A_0''$ and $A_1''$ are some linear combination of products of 
		\[
		\eta(z) + O''(|z-\zeta|), \; (Q'')^{-1}, \; D_\zeta \ov{\Phi(\zeta,z)}, \; D_z \ov{\Phi(\zeta,z)}, \; D_\zeta \rho(\zeta).  
		\] 
		Likewise, the integrals in \re{F2_ibp_ktimes} can be written as 
		\[
		\int_{bD} \frac{D^k f(\zeta) A'_0(z,\zeta)}{\Phi^{n+1}(z, \zeta) }  \, d\si(\zeta), \quad 
		\int_D \frac{D^k f(\zeta) A'_1(z,\zeta)}{\Phi^{n+2}(z, \zeta) }  \, dV(\zeta), 
		\] 
		where $A_0'$ and $A_1'$ are linear combination of products of 
		\[
		\eta(\zeta) + O'(|z-\zeta|), \; (Q')^{-1}, \; D_\zeta \Phi(z,\zeta), \; D_z \Phi(z, \zeta), \; D_\zeta \rho( \zeta),  
		\] 
		with coefficients identical to the linear combination $A_0''$ and $A_1''$, respectively. 
		In view of \re{Dgm_Kf}, it suffices to estimate the difference: 
		\[
		\int_{bD} D^k f(\zeta) \left( \frac{A_0''(z,\zeta)}{\ov{\Phi}^{n+1}(\zeta,z)} - \frac{A_0'(z,\zeta)}{\Phi^{n+1}(z,\zeta)} \right) 
		\, d\si(\zeta), \quad 
		\int_D D^k f(\zeta) \left( \frac{A_1''(z,\zeta)}{\ov{\Phi}^{n+2}(\zeta,z)} - \frac{A_1'(z,\zeta)}{\Phi^{n+2}(z,\zeta)} \right) 
		\, dV(\zeta). 
		\]
		We shall again estimate only the domain integral as the proof for the boundary integral is similar. By the expression for $\eta$ and \rl{Lem::Phi_grad_diff}, we have
		\begin{gather*}
		| \eta(z) - \eta(\zeta)| \lesssim |\rho|_3 |\zeta -z|, \quad 
		|D_\zeta \ov{\Phi(\zeta,z)} - D_\zeta \Phi(z,\zeta) | \lesssim |\rho|_3 |\zeta-z|
		\\ 
		|D_z \ov{\Phi(\zeta,z)} - D_z \Phi(z,\zeta) | \lesssim |\rho|_3 |\zeta-z|, \quad 
		|Q''(z,\zeta) - Q'(z,\zeta)| \lesssim |\rho|_3 |\zeta-z|. 
		\end{gather*}
		By procedure similar to the estimates of the $I_2$ integral in the $k=0$ case, we can prove the following estimate
		\begin{align*}
		\left| \int_D D^k f(\zeta) \left( \frac{A_1''(z,\zeta)}{\ov{\Phi}^{n+2}(\zeta,z)} - \frac{A_1'(z,\zeta)}{\Phi^{n+2}(z,\zeta)} \right)
		\, dV(\zeta) \right| \lesssim |\rho|_3 |f|_k  \del(z)^{-\yh}. 
		\end{align*} 
		Consequently we conclude that in this case 
		\[
		| D^\gm_z \Kc_1'f(z)| \lesssim |\rho|_3 |f|_k \del(z)^{-\yh}. 
		\]
		Finally combining the results from all cases we have shown that  
		\[ 
		| D_z^{k+1} \Kc'_1 f (z) | \lesssim  	 \begin{cases} 
		| f|_k |\rho|_{k+3+\all} \del(z)^{-1 +\all }, & \text{if $ 0 < \all \leq \yh$};  \\ 
		| f|_k |\rho|_{k+3}\del(z)^{-\yh}, & \text{if $ \yh \leq \all \leq 1$}.  
		\end{cases}
		\] 
		By \rl{Lem::H-L}, $\Kc_1' f \in C^{k+\min \{ \all, \yh \}} (\ov D)$.  Combined with earlier estimates for $\Kc_1'' f$ and $\Kc_0 f$, the proof of \rp{Prop::K_bdd_intro} is now complete.  
	
	\begin{prop} \label{Prop::hf} 
		Let $D$ be a strictly pseudoconvex domain with $C^3$ boundary. Let $f$ be a function in $C^1(\ov \Om)$ such that $\dbar f \in C^1(\ov \Om)$. Then the following formula holds 
		\begin{align*}
		f(z) &= \Lc f(z)+ \int_D S_z(\dbar_z \dbar_\zeta N)(z,\cdot) \we f- \int_D N(z,\cdot) \we \dbar f
		\\ & \quad + \int_{b D} \Om^{01}_{0,0}(z,\cdot) \we \dbar f
		+ \int_D \Om^0_{0,0}(z,\cdot) \we \dbar f, \quad z \in D .
		\end{align*}
		Here $N$ and $\Lc$ are given by formulae \re{N_def} \re{Lc_def}. 
	\end{prop} 
	\begin{proof} 
		Starting with the Bochner-Martinelli formula (see for example \cite[Theorem 2.2.1]{C-S01}): 
		\[
		f(z)= \int_{bD} \Om^0_{0,0}(z,\zeta) \we f(\zeta) 
		+ \int_D \Om^0_{0,0}(z,\zeta) \we \dbar f, \quad z \in D. 
		\] 
		By \re{Kop_N}, \re{ker_L_def} and Stokes' theorem, we have 
		\begin{align*}
		f(z) &= \int_{bD} N(z,\cdot) \we f 
		+ \int_{bD} \dbar_\zeta \Om^{01}_{0,0}(z,\cdot) \we f + \int_D \Om^0_{0,0}(z,\zeta) \we \dbar f
		\\ &= 
		\int_{D} \dbar_\zeta N(z,\cdot) \we f 
		- \int_D N(z,\cdot) \we \dbar f + \int_{bD} \Om^{01}_{0,0}(z,\cdot) \we \dbar f + \int_D \Om^0_{0,0}(z,\zeta) \we \dbar f 
		\\ &= \Lc f(z) + \int_D S_z(\dbar_z \dbar_\zeta N)(z,\cdot) \we f
		- \int_D N(z,\cdot) \we \dbar f 
		\\ &\qquad + \int_{bD} \Om^{01}_{0,0}(z,\cdot) \we \dbar f + \int_D \Om^0_{0,0}(z,\zeta) \we \dbar f. \qedhere 
		\end{align*} 
	\end{proof} 
	\begin{prop} \label{Prop::main_est} 
		Let $D$ be a bounded strictly pseudoconvex domain with $C^{k+3+\all}$ boundary, with $0 < \all \leq 1$. Suppose $f$ is orthogonal to the Bergman space $H^2(D)$, is $C^\infty$ in $D$ and is holomorphic in $D \sm D_{-\del}$, for some $\del>0$. Then $f \in  C^{k+ \min \{ \all, \yh \}}(\ov D)$. 
	\end{prop}  
	Here we recall the notation. 
	\[
	D_{-\del}:= \{ z \in D: \rho(z) < -\del \}. 
	\]
	\begin{proof}
		Let $\Pc$ be the Bergman projection for $D$. By assumption $\Pc f \equiv 0$. 
		By \rp{KP_prop}, $\Lc^\ast f = (I + \Kc) \Pc f \equiv 0$, which implies that $\Kc f = \Lc^\ast f - \Lc f = -\Lc f$. Consequently by \rp{Prop::hf} and the assumption that $\dbar f \equiv 0$ on $b D$, 
		\begin{equation} \label{f_int_eqn}  
		\begin{aligned} 
		f(z)  &= - \Kc f(z) + \int_D S_z(\dbar_z \dbar_\zeta N)(z,\cdot) \we f - \int_D N(z, \cdot) \we \dbar f
		\\ & \qquad + \int_{b D} \Om^{01}_{0,0}(z,\cdot) \we \dbar f + \int_D \Om^0_{0,0}(z, \cdot) \we \dbar f 
		\\ &= - \Kc f(z) + \int_D S_z(\dbar_z \dbar_\zeta N)(z,\cdot) \we f - \int_D N(z, \cdot) \we \dbar f
		+ \int_D \Om^0_{0,0}(z, \cdot) \we \dbar f, \quad z \in D. 
		\end{aligned}
		\end{equation} 
		Here the kernels $\Om^0_{0,0}$ and $N$ are given by formulae \re{Om0_def} and \re{N_def} on $D$: 
		\eq{Nexp}
		N(z,\zeta) = \frac{1}{(2 \pi \sqrt{-1} )^n} \frac{1}{[G(z,\zeta) - \rho(\zeta)]^n} \left( \sum g_1^i(z,\zeta) d \zeta_i \right) \we \left( \sum_i \dbar_\zeta g_1^i(z,\zeta) \we d \zeta_i \right)^{n-1}; 
		\eeq
		\[
		\Om_{0,0}^0(z,\zeta) = \frac{1}{(2\pi \sqrt{-1})^n} \frac{1}{|\zeta -z|^{2n}} \sum_{i=1}^n  
		(\ov{\zeta_i - z_i}) d \zeta_i \we \left( \sum_{j=1}^n (
		d \ov \zeta_j - d \ov z_j ) \we d \zeta_j \right)^{n-1}, 
		\]  
		where $G$ and $g_1$ are given by expressions \re{G_def} and \re{g1_def}. 
		We can rewrite \re{f_int_eqn} as
		\begin{equation} \label{f_int_eqn_h} 
		f + \Kc f = h := h_1 + h_2 + h_3, 
		\end{equation} 
		where we denote
		\[
		h_1(z) := \int_D S_z(\dbar_z \dbar_\zeta N)(z,\cdot) \we f, 
		\quad 
		h_2(z) := - \int_D N(z,\cdot) \we \dbar f,
		\quad
		h_3 (z) := \int_D \Om^0_{0,0} (z,\cdot) \we \dbar f.
		\] 
		We show that each $h_i$ defines a function in $C^\infty(\ov D)$. By the first statement in \re{L0_int_est}, we have $h_1 \in C^\infty( \ov D)$. 
		For $h_2$, note that the functions $G(z,\zeta), g_1(z,\zeta)$ are $C^\infty$ in $z$, and the following estimate (see \re{G-rho_lbd_2}) holds 
		\[
		G(z,\zeta) - \rho(\zeta) \geq c(-\rho(z) - \rho(\zeta) + |z-\zeta|^2), \quad z, \zeta \in \ov D.  
		\] 
		In particular, for $\zeta \in \supp (\dbar f)$, i.e. $\zeta \in D_{-\del}$, the function $G(z,\zeta) - \rho(\zeta)$ is bounded below by some positive constant for all $z \in \ov D$. Hence in view of \re{Nexp}, $h_2 \in C^\infty(\ov D)$. To see that $h_3 \in C^\infty(\ov D)$, we note that by assumption $\dbar f \in C^\infty_c( D)$, and the argument is done using integration by parts. 
		
		Now, by \rp{Prop::K_bdd_intro}, $\Kc$ is a compact operator on the Banach space $C^k(\ov D)$. Thus by the Fredholm alternative, either $I + \Kc$ is invertible or $\ker (I + \Kc)$ is non-empty.  Suppose $f \in \ker(I + \Kc)$; then $f = - \Kc f$ and $\sqrt{-1} \Kc f = - \sqrt{-1}f$. If $f \neq 0$, this would imply that $-\sqrt{-1}$ is an eigenvalue of the operator $\sqrt{-1} \Kc$, which is impossible since $\sqrt{-1} \Kc$ is self-adjoint and have only real eigenvalues. Therefore we conclude that $f \equiv 0$, and $\ker (I+ \Kc) = \emptyset$. This implies $I+\Kc$ is an invertible operator on the space $C^k(\ov D)$. 
		
	Applying this to equation \re{f_int_eqn_h} and $h \in C^\infty(\ov \Om)$, we obtain $f \in C^k(\ov D)$. By \rp{Prop::K_bdd_intro}, we have $\Kc f \in C^{k+ \min\{ \all, \yh \}} (\ov D)$. Hence $f = - \Kc f + h \in C^{k+ \min \{ \all, \yh \}} (\ov D)$. 
	\end{proof}
We can now finally prove \rt{Thm::Bker_intro}. 
\\ 
\nid \tit{Proof of \rt{Thm::Bker_intro}.} 
Fix $w_0 \in D$, we can write the $B(\cdot,w_0) = \Pc \var$, where $\var \in C^\infty_c (D)$. Applying \rp{Prop::main_est} to $f = \Pc \var - \var$ we get $\Pc \var - \var \in C^{k+ \min \{ \all, \yh \}}(\ov D)$. Hence $\Pc \var \in C^{k+ \min \{ \all, \yh \}}(\ov D)$.  \qed 
	
	\section{Proof of \rt{Thm::Bproj_intro}}
In this section we prove \rt{Thm::Bproj_intro}, which will also follow from \rp{Prop::K_bdd_intro}.   
	First we need an approximation lemma. 
	\begin{lemma} \label{Lem::approx} 
	Let $D$ be a bounded Lipschitz domain in $\R^N$.  Suppose $f \in C^{k+\beta}(\ov D)$, where $k$ is a non-negative integer and $0 < \beta < 1$. Then there exists a family $\{ f_\ve \}_{\ve>0} \in C^\infty(D) \cap C^{k+\beta}(\ov D)$ such that $f_\ve$ converges to $f$ uniformly as $ \ve \to 0$. Furthermore, $ |f_\ve|_{k+\beta}$ is uniformly bounded by $|f|_{k+\beta}$.   
	\end{lemma}
\begin{rem} \label{Rem::Holder_cvg} 
  Let $f_\ve$ be constructed as above. It follows from \cite[Prop 2.3]{Shi23} that $f_\ve$ converges to $f$ in $|\cdot|_{\tau} $, for any $0 \leq \tau < k+\beta$.   
\end{rem} 
	\begin{proof}
		It suffices to take $D$ as a special Lipschitz domain of the form $ \om = \{ x \in \R^N: x_N> \psi(x_1, \dots, x_{n-1}), |\psi|_{L^\infty} \leq C  \} $, as the general case follows by standard partition of unity argument. There exists some cone $K$ such that for any $x \in \om$, $x + K \subseteq \om$. 
		Let $\phi$ be a $C^\infty$ with compact supported in $-K$ and such that $\phi \geq 0$ and $\int_{\R^N} \phi =1$. Let $\phi_\ve = \frac{1}{\ve^n} \phi(\frac{x}{\ve})$. Then we can define for $x \in \om$ the function 
		\[
		f_\ve(x) = f \ast \phi_\ve (x) = \int_{-K} f(x- \ve y) \phi(y) \, dV(y) . 
		\]
		It is clear that $f_\ve \in C^\infty (D)$ and
		\begin{align*}
		| f_\ve(x) - f(x)|  &= \left| \int_{-K} [f(x- \ve y) -f(x)] \phi (y) \, dV(y)  \right| 
		\\ &\leq \int_{-K} \left|f(x-\ve y) - f(x) \right| \phi(y) \, dV(y) 
		\\ &\leq |f|_\beta \ve^\beta \int_{-K} y^\all \phi(y) \, dV(y) \lesssim |f|_\beta \ve^\beta. 
		\end{align*}
		Hence $f_\ve $ converges to $f$ uniformly in $\om$. Let $x_1, x_2 \in \om$. Then for all $\ve>0$, 
		\begin{align*}
		\left|  f_\ve(x_1) - f_\ve(x_2) \right| 
		&= \left| \int_{-K} [ f(x_1 - \ve y)  -  f(x_2 - \ve y) ] \phi(y) \, dV(y)   \right| 
		\\ &\leq |f|_\beta |x_1 - x_2|^\beta. 
		\end{align*}
		Accordingly $|f_\ve|_\beta$ is uniformly bounded by $|f|_\beta$. This proves the case $ k=0$. For $k\geq 1$ the proof is similar and we leave the details to the reader.  
	\end{proof}
	
	\begin{lemma} \label{Lem::L_opt_k+beta_est} 
		Let $D$ be a strictly pseudoconvex domain in $\C^n$ with $C^{k+3}$ boundary, where $k$ is a non-negative integer. Then for any $ \beta >0$, 
		\[
		| \Lc f |_{k} \lesssim |f|_{k+\beta}. 
		\]
	\end{lemma}
	\begin{proof} 
		Write $D^k_z \Lc f(z)$ as a linear combination of 
		\eq{Lf_kder} 
		\int_D f(\zeta) \frac{W(z,\zeta)}{\Phi^{n+1+\mu}(z,\zeta)} \, dV(\zeta), \quad \mu \leq k, 
		\eeq
		where $W(z,\zeta)$ is some linear combination of products of  
		\[
		D^{\tau_1}_z [l(\zeta) + O'(|z-\zeta|) ], \quad D_z^{\tau_2} \Phi(z,\zeta), \quad \mu_1, \mu_2 \leq k. 
		\]
		Applying integration by parts formulae \re{ibp_1} and \re{ibp_bdy}  iteratively to the integral \re{Lf_kder} until it can be written as a linear combination of 
		\eq{Lf_kder_ibp}
		\int_{bD} D^{\eta_0} f(\zeta)  \frac{ W_0(z,\zeta)}{ \Phi^{n+1}(z,\zeta)} \, d\si(\zeta),
		\quad 
		\int_D D^{\mu_0} f(\zeta)  \frac{ W_1(z,\zeta)}{ \Phi^{n+1}(z,\zeta)} \, dV(\zeta), 
		\quad \eta_0, \mu_0 \leq k. 
		\eeq
		Here $W_0$ and $W_1$ are some linear combination of 
		\eq{Lf_kder_ibp_lc} 
		D_\zeta^{\mu_1} D^{\tau_1}_z [l(\zeta) + O'(|z-\zeta|) ], \; D^{\mu_2}_\zeta [ (Q')^{-1} ], 
		\; D^{ \mu_3 +1}_\zeta \Phi(z,\zeta), \; D^{\mu_4}_\zeta D_z^{\tau_2} \Phi(z,\zeta), \; D^{\mu_5+1} _\zeta \rho(\zeta),
		\eeq
		with $\mu_i \leq k$, $0 \leq i \leq 5$ and $\sum_{i=0}^5 \mu_i \leq k$.  
		We shall only estimate the domain integral in \re{Lf_kder_ibp}, as the proof of the boundary integral is similar. 
		In view of \re{Lf_kder_ibp_lc}, we can write $W_1 (z,\zeta) = Y_1(z,\zeta) + Y_2 (z,\zeta)$, where $|Y_1(z,\zeta)|  \lesssim |\rho|_{k+2}  $, and $|Y_2(z,\zeta)| \lesssim |\rho|_{k+3} |\zeta-z|$.  Write
		\[
		\int_D D^{\mu_0} f(\zeta)  \frac{ W_1(z,\zeta)}{ \Phi^{n+1}(z,\zeta)} \, dV(\zeta) 
		=  \int_D D^{\mu_0} f(\zeta)  \frac{ Y_1(z,\zeta)}{ \Phi^{n+1}(z,\zeta)} \, dV(\zeta)
		+  \int_D D^{\mu_0} f(\zeta)  \frac{ Y_2(z,\zeta)}{ \Phi^{n+1}(z,\zeta)} \, dV(\zeta). 
		\] 
		The $Y_2$ integral is bounded by 
		\begin{align*}
		\left| \int_D D^{\mu_0} f(\zeta)  \frac{ Y_2(z,\zeta)}{ \Phi^{n+1}(z,\zeta)} \, dV(\zeta) \right| 
		\lesssim |\rho|_{k+3} |f|_k \int_D \frac{|\zeta-z|}{| \Phi(z,\zeta)|^{n+1}} \lesssim  |\rho|_{k+3} |f|_k, 
		\end{align*} 
	where we used \rl{zeta-z_Phi_int_2}. For the $Y_1 $ integral we use the assumption that $ f \in C^{k+\beta}$, $\beta >0$.  
		\begin{align*}
		\int_D D^{\mu_0} f(\zeta)  \frac{ Y_1(z,\zeta)}{ \Phi^{n+1}(z,\zeta)} \, dV(\zeta) 
		&=   \int_D [D^{\mu_0} f(\zeta) -  D^{\mu_0} f(z) ] \frac{ Y_1(z,\zeta)}{ \Phi^{n+1}(z,\zeta)} \, dV(\zeta) + D^{\mu_0} f(z) \int_D \frac{Y_1(z,\zeta) \, dV(\zeta)}{\Phi^{n+1} (z,\zeta)}. 
		\end{align*} 
		The first integral on the right-hand side is bounded by 
		\[ 
		|\rho|_{k+2} |f|_{k+\beta} \int_D \frac{|\zeta-z|^\beta \,  dV(\zeta)}{|\Phi^{n+1} (z,\zeta)|} 
		\lesssim |\rho|_{k+2} |f|_{k+\beta} .
		\]
		For the other integral, since $Y_1$ involves derivatives of $\rho$ up to order $k+2$, we can apply integration by parts and take one more derivative of $\rho$ against $\zeta$. The resulting integrals are bounded by $|\rho|_{k+3} $ up to a constant. Summing up the estimates we have
		\[
		\left| \int_D D^{\mu_0} f(\zeta)  \frac{ Y_1(z,\zeta)}{ \Phi^{n+1}(z,\zeta)} \, dV(\zeta) \right|  \lesssim |\rho|_{k+3} |f|_{k+\beta}. 
		\] 
		Consequently this shows that $|D^k_z \Lc f(z)| \lesssim |\rho|_{k+3} |f| _{k+\beta}$, finishing the proof. 
	\end{proof}
	We are now ready to prove \rt{Thm::Bproj_intro}. 
	\begin{prop} \label{Prop::Bproj} 
		Let $D$ be a strictly pseudoconvex domain in $\C^n$. Let $k$ be a non-negative integer, and $0 < \all, \beta \leq 1$. 
		\begin{enumerate}[(i)]
			\item 
			Suppose $bD \in C^{k+3}$. Then $\Lc$ defines a bounded operator from $C^{k+\beta}(\ov D)$ to $C^{k+ \frac{\beta}{2}}(\ov D)$. 
		\item 
Suppose $bD \in C^{k+3+\all}$. Then $\Pc, \Lc^\ast$ define bounded operators from $C^{k+\beta}(\ov D)$ to $C^{k+ \min \{\all, \frac{\beta}{2} \} }(\ov D)$. 
		\end{enumerate}
	\end{prop} 
	\begin{proof}
(i) We first prove the statement for $\Lc$ and we begin by considering the case $k=0$. Assume first that $0 < \beta <1$. Let $f \in C^\beta (\ov D)$ and $\{ f_\ve \}_{\ve >0}$ be the functions constructed in \rl{Lem::approx}. In particular, we have 
\begin{enumerate}
    \item \label{fve_Holder}  
    $f_\ve \in C^\infty(D) \cap C^{\beta}(\ov D)$; 
    \item \label{fve_cvg} 
    $|f_\ve-f|_{\eta} \to 0$, for any $0 \leq \eta < 
  \beta$ (\rrem{Rem::Holder_cvg}).  
\end{enumerate}
 We claim that for each $\Lc f_\ve \in C^{\frac{\beta}{2}} (\ov D)$ with $ |\Lc f_\ve|_{\frac{\beta}{2}}$ uniformly bounded by some constant $C_0$. Assuming the claim holds, then for any $z_1, z_2 \in D$, we have 
		\begin{equation} \label{Lf_diff}  
		\begin{aligned} 
		| \Lc f(z_1) -\Lc f(z_2)| &\leq | \Lc f(z_1) - \Lc f_\ve (z_1) | + | \Lc f_\ve (z_1)- \Lc f_\ve (z_2) | + | \Lc f(z_2) - \Lc f_\ve (z_2) |
		\\ &\leq 2 |\Lc (f-f_\ve)|_0 + |\Lc f_\ve |_{ \frac{\beta}{2}} |z_1 - z_2|^{\frac{\beta}{2}}  
		\\ &\leq  2 |\Lc (f-f_\ve)|_0  + C_0 |z_1 - z_2|^\frac{\beta}{2}.  
		\end{aligned}  
		\end{equation}
		Now, given a function $g \in C^\eta(\ov D)$ with $\eta >0$, using the reproducing property of $\Lc$ we have 
		\begin{equation} \label{Lg_sup_norm_est}
		\begin{aligned}  
		|\Lc g(z) | &= \left| \int_D [ g(\zeta) - g(z) ] L(z,\zeta) \, dV(\zeta) + g(z) \right| 
		\\ &\lesssim |\rho|_2 |g|_{\eta} \int_D  \frac{|\zeta-z|^\eta}{|\Phi(z,\zeta)|^{n+1}} \, dV(\zeta) + |g|_0 \lesssim (| \rho|_2 +1) |g|_\eta, 
		\end{aligned} 
		\end{equation}
  where in the last inequality we applied \rl{Lem::int_est}. 
  
 Applying \re{Lg_sup_norm_est} with $g = f-f_\ve$ and using property 
 \re{fve_cvg} from above, we get $|\Lc ( f- f_\ve)|_0 \to 0$ as $\ve \to 0$. It follows from 
  \re{Lf_diff} that $|\Lc f(z_1) - \Lc  f(z_2)| \leq C_0 | z_1 - z_2|^{\frac{\beta}{2}}$. This shows that $\Lc f \in C^{\frac{\beta}{2}} (\ov D)$. 
  
It remains to prove the claim, namely, $|\Lc f_\ve |_{\frac{\beta}{2}}$ is bounded by some constant $C_0$ independent of $\ve$. To this end, we will show that $|\Lc f_\ve |_{\frac{\beta}{2}} \leq C_0' | f_\ve|_{\beta} $, where $C_0'$ depends only on $|\rho|_3$. Since $| f_\ve |_{\beta} \leq |f|_\beta$, this proves the claim. 
		
For $f_\ve \in C^\infty(D) \cap C^\beta(\ov D)$, we have
\[
  f_\ve(z) - \Lc f_\ve(z) = \int_D \left[ f_\ve(z) - f_\ve(\zeta) \right] L(z,\zeta) \, dV(\zeta), 
\] 
where we used the reproducing property of kernel $L$: $\int_D L(z,\zeta) dV(\zeta) \equiv 1 $.  
Then 
\[
  \DD{f_\ve}{z_i}(z) - \DD{\Lc f_\ve}{z_i}(z) 
= \int_D \DD{f_\ve}{z_i}(z) L(z,\zeta) \, dV(\zeta)
+ \int_D \left[ f_\ve(z) - f_\ve(\zeta) \right] \DD{L}{z_i}(z,\zeta) \, dV(\zeta) . 
\]
The first term on each side cancels out, which leaves us with 
\begin{align*}
&\DD{\Lc}{z_i}(z,\zeta) \, dV(\zeta)
 = \int_D [f_\ve(\zeta) - f_\ve(z)] \pp{L}{z_i} (z,\zeta) \, dV(\zeta)
\\ &= \int_D \left[ f_\ve (\zeta) - f_\ve (z) \right] \left[ \frac{\pa_{z_i} \left[ l(\zeta) + O'(|z-\zeta|) \right] }{\Phi^{n+1}(z,\zeta)} - (n+1)  
 \frac{ [ l(\zeta) + O'(|z-\zeta|) ]\pa_{z_i} \Phi (z,\zeta)}{\Phi^{n+2}(z,\zeta)} \right] \, dV(\zeta) . 
		\end{align*} 
For $\ov z_i$ derivatives we have a similar expression. By estimate \re{O'_est} and \re{Phi_lbd}, we obtain
		\begin{align*}
	| \na \Lc f_\ve(z) | &\lesssim |\rho|_3 |f_\ve|_\beta  \left( \int_D \frac{|z-\zeta|^\beta }{ |\Phi(z,\zeta)|^{n+1} } \, dV(\zeta) + \int_D \frac{|z-\zeta|^\beta }{|\Phi(z,\zeta)|^{n+2}} \, dV(\zeta)
		\right) \\ &\lesssim |\rho|_3 |f_\ve|_\beta \left(
		\int_D \frac{|z-\zeta|^\beta }{ |\Phi(z,\zeta)|^{n+1} } \, dV(\zeta) 
		+  \int_D \frac{dV(\zeta) }{|\zeta-z|^{2-\beta} \Phi(z,\zeta)^{n+1}} 
		\right) 
		\\ &\lesssim |\rho|_3 |f_\ve|_\beta 
\left( 1+ \del(z)^{-1+\frac{\beta}{2}} \right) , 
		\end{align*} 
where in the last step we applied \rl{Lem::int_est}. It follows by Hardy-Littlewood lemma that $\Lc f_\ve \in C^{\frac{\beta}{2}}(\ov D)$ and $|\Lc f_\ve|_{\frac{\beta}{2}}$ is bounded by $C_0' |f_\ve|_\beta$, where $C_0'$ depends only on $|\rho |_3$.  Combined with the earlier argument, this proves (i) for $k=0$ and $0 < \beta <1$. If $k=0$ and 
$\beta =1$, we can repeat the above proof without doing the approximation, obtaining in the end 
\[
  |\na \Lc f(z)| \lesssim |\rho|_3 |f|_1 \left(  (1+ \del(z)^{-\yh} \right), \quad z \in D. 
\]
Hence by Hardy-Littlewood lemma, $\Lc f \in C^{\yh}(\ov D)$.   		

Next we consider the case $k \geq 1$. Suppose $f \in C^{k+\beta} (\ov D)$, for $0<\beta < 1$. As before we first construct $\{ f_\ve \}_{\ve >0}$ such that
\begin{enumerate} 
    \item  
    $f_\ve \in C^\infty(D) \cap C^{k+\beta}(\ov D)$; 
    \item \label{fve_cvg_k+beta}   
    $|f_\ve-f|_{\eta} \to 0$, for any $0 \leq \eta < k+\beta$. 
\end{enumerate} 

We claim that $|\Lc f_\ve|_{k+\frac{\beta}{2}}$ is bounded uniformly by some constant $C_0$. Assuming the validity of the claim, for $ z_1, z_2 \in D$ and $\ell \leq k$, we have 
\begin{equation} \label{Lf_der_diff}
\begin{aligned} 
  | D^\ell \Lc f(z_1) - D^\ell \Lc f(z_2)| &\leq | D^\ell \Lc f(z_1) - D^\ell \Lc f_\ve (z_1) | + | D^\ell \Lc f_\ve (z_1)- D^\ell  \Lc f_\ve (z_2) | 
  \\ &\quad + | D^\ell \Lc f(z_2) - D^\ell \Lc f_\ve (z_2) |
\\ &\leq 2 |\Lc (f-f_\ve)|_\ell + |\Lc f_\ve |_{k+ \frac{\beta}{2}} |z_1 - z_2|^{\frac{\beta}{2}} 
	\\ &\leq  2 |\Lc (f-f_\ve)|_k  + C_0 |z_1 - z_2|^{\frac{\beta}{2}}.  
\end{aligned}
\end{equation} 
		As before we want to show that $ |\Lc (f-f_\ve)|_k \to 0$ as $ \ve \to 0$. 
		Here the estimate is more subtle since $D \Lc g = \Lc Dg$ does not hold and thus one cannot estimate as easily as in \re{Lg_sup_norm_est}.  Instead we apply \rl{Lem::L_opt_k+beta_est} to get 
		\eq{Lg_der_est_recall} 
		| \Lc (f - f_\ve) |_k \lesssim | f-f_\ve|_{k+ \tau}  , \quad \text{ for any $\tau >0$}. 
		\eeq
By property \re{fve_cvg_k+beta} above, we have $|f-f_\ve|_\eta \to 0 $ for any $\eta < k+\beta$. Hence  \re{Lg_der_est_recall} implies $|\Lc (f-f_\ve) |_k \to 0 $. Letting $\ve \to 0$ in \re{Lf_der_diff}, we get $\Lc f \in C^{k+\frac{\beta}{2}} (\ov D) $, which proves the reduction. 
		
To finish the proof it remains to show that there exists a constant $C_0'>0$ (which we will show depends only on $|\rho|_{k+\beta}$) such that $|\Lc f_\ve |_{k+ \frac{\beta}{2}} \leq C_0' |f_\ve|_{k+\beta}$. Then by \rl{Lem::approx} we get $|\Lc f_\ve |_{k+ \frac{\beta}{2}}  \leq C_0' |f_\ve|_{k+\beta} \leq C_0' | f |_{k+\beta}$. We have
 \begin{align*}
D^{k+1}_z \left[ f_\ve(z) - \Lc f_\ve(z) \right]
&= D^{k+1}_z \int_D [f_\ve(z) - f_\ve(\zeta)] L(z,\zeta) \, dV(\zeta)
		\\ &= \int_D D^{k+1}_z f_\ve(z) L(z,\zeta) \, dV(\zeta)
		+ \int_D \sum_{\substack{\gm_1+ \gm_2 = k+1 \\ 1 \leq \gm_2 \leq k} } D^{\gm_1}_z f_\ve(z)  D^{\gm_2}_z L(z,\zeta) \, dV(\zeta) 
		\\ & \quad + \int_D (f_\ve(z) - f_\ve(\zeta)) D^{k+1}_z L(z,\zeta) \, dV(\zeta). 
		\end{align*}
		The first integral is equal to $D^{k+1}_z f(z)$. Hence
		\begin{align*}
		D^{k+1}_z  \Lc f_\ve(z) &= - \int_D \sum_{\substack{\gm_1+ \gm_2 = k+1 \\ 1 \leq \gm_2 \leq k} } D^{\gm_1}_z f_\ve(z)  D^{\gm_2}_z L(z,\zeta) \, dV(\zeta) + \int_D [f_\ve(\zeta) - f_\ve(z)] D^{k+1}_z L(z,\zeta) \, dV(\zeta):
		\\ &:= I_1 + I_2, 
		\end{align*}
		where we denote the first and second integral by $I_1$ and $I_2$, respectively. 
		For $I_1$, we can write it as  a linear combination of integrals of the form
		\begin{align} \label{Lf_der_I1}  
		D_z^{\mu_0} f_\ve(z) \int_D \frac{ W(z,\zeta )}{\Phi^{n+1+\mu_1}(z,\zeta)} \, dV(\zeta) , \quad \mu_0, \mu_1 \leq k, 
		\end{align} 
		where $W$ is some linear combination of $D^{\mu_2}_z[l(\zeta) + O'(|z-\zeta|)]$ and $D_z^{\mu_3} \Phi(z,\zeta)$ with  $\mu_2, \mu_3 \leq k$. We apply integration by parts formulae \re{ibp_1} and \re{ibp_2} iteratively to the integral in \re{Lf_der_I1} until it can be written as a linear combination of 
		\eq{Lf_der_I1_ibp} 
		\int_{bD} \frac{W_0( z,\zeta)}{\Phi^n(z,\zeta)} \, d\si(\zeta) , \quad 
		\int_D \frac{W_1( z,\zeta)}{\Phi^{n+1} (z,\zeta)} \, dV(\zeta). 
		\eeq 
		Here $ W_0, W_1$ are linear combinations of products of 
		\[
		D_\zeta^{\tau_1} D_{z_i}^{\mu_2} [l (\zeta) + O'( |z-\zeta|)] , 
		\;  D_\zeta^{\tau_2}[( Q')^{-1}], \; D^{\tau_3+1}_{\zeta} \Phi(z,\zeta), 
		\; D_\zeta^{\tau_4} D_z^{\mu_3} \Phi(z,\zeta), 
		\; D_{\zeta}^{\tau_5 +1} \rho(\zeta), 
		\]
		with $\tau_i  \leq k$, $1 \leq i \leq 5$. 
		Note that all these quantities are bounded by some constant multiple of $|\rho|_{k+3}$.  
		It follows that the integrals in \re{Lf_der_I1_ibp} and hence $I_1$ is bounded by 
  \begin{equation} \label{Lf_I1_est} 
		|I_1| \lesssim |f_\ve|_k |\rho|_{k+3} \left( \int_{bD} \frac{d \si(\zeta)}{|\Phi(z,\zeta)|^n}  + \int_D \frac{dV(\zeta)}{|\Phi(z,\zeta)|^{n+1}}  \right) 
		\lesssim  |f_\ve|_k |\rho|_{k+3} (1 + \log \del(z) ), 
  \end{equation} 
  where we applied \rl{Lem::Phi_int_est}.   
		The integral $I_2$ can be written as a linear combination of integrals of the form: 
		\eq{Lf_der_I2}  
		\int_D \frac{[f_\ve(\zeta) - f_\ve(z)] W(z,\zeta)}{\Phi^{n+2+\mu}(z,\zeta)} \, dV(\zeta), \quad \mu \leq k. 
		\eeq 
		Here $W(z,\zeta)$ is some linear combination of 
		\[
		D^{\tau_0}_z [ l(\zeta) + O'(|z -\zeta |)]  , \quad D^{\tau_1}_z \Phi(z,\zeta), \quad \tau_0, \tau_1 \leq k+1. 
		\]
		
		If $ \mu \leq k-1$ we can integrate by parts and estimate just like $I_1$ to show that $ |I_2| \lesssim |f_\ve|_k |\rho|_{k+3} (1 + \log|\del(z)|)$.  If $\mu = k$, we apply integration by parts formulae \re{ibp_1} and\re{ibp_bdy} until the integral \re{Lf_der_I2} can be expressed as a linear combination of integrals of the form 
		\eq{Lf _der_I2_ibp} 
		\int_{bD}  \frac{D^{\eta_0}_\zeta f_\ve (\zeta) A_0(z,\zeta)}{\Phi^{n+1} (z,\zeta)} \, d\si(\zeta), \quad \int_D \frac{ D^{\mu_0} f_\ve(\zeta) A_1(z,\zeta)}{\Phi^{n+2} (z,\zeta)} \, dV(\zeta), \quad \eta_0, \mu_0 \leq k. 
		\eeq 
		Here $A_0, A_1$ are linear combination of products of 
		\eq{Lf_der_A0A1}
		D^{\mu_1}_\zeta[ l(\zeta) + O'(|z -\zeta |)] , \; 
		D^{\mu_2}_\zeta [(Q')^{-1}], 
		\; 
		D^{\mu_3 + 1}_\zeta \Phi(z,\zeta), 
		\; 
		D^{\mu_4}_\zeta D_z \Phi(z,\zeta), 
		\; 
		D^{\mu_5+1}_\zeta \rho(\zeta),  
		\eeq 
		where $\mu_i \leq k$, and $\sum_{i=0}^5 \mu_i = k$. We now use the fact that $\rho \in C^\infty(D) \cap C^{k+3}(\ov D)$ satisfies the estimate 
		\[
		|D^j_z \rho(z) | \lesssim C_j |\rho|_{k+3} \left(1 + \del(z)^{k+3 -j} \right), \quad j =0,1,2, \dots. 
		\] 
We shall only estimate the domain integral in \re{Lf _der_I2_ibp}, as the estimate for the boundary integral is similar. In view of \re{Lf_der_A0A1} we can write $A_1(z,\zeta) = X_1(z,\zeta) + X_2(z,\zeta)$, where $|X_1(z,\zeta) | \lesssim |\rho|_{k+2}$ and $|X_2(z,\zeta) | \lesssim |\rho|_{k+3} |\zeta-z|$. Write 
		\[
		\int_D \frac{ D^{\mu_0} f_\ve(\zeta) A_1(z,\zeta)}{\Phi^{n+2} (z,\zeta)} \, dV(\zeta) 
		=   \int_D \frac{ D^{\mu_0} f_\ve(\zeta) X_1(z,\zeta)}{\Phi^{n+2} (z,\zeta)} \, dV(\zeta)  +   \int_D \frac{ D^{\mu_0} f_\ve(\zeta) X_2(z,\zeta)}{\Phi^{n+2} (z,\zeta)} \, dV(\zeta). 
		\]
By estimates \re{Phi_lbd} and \re{zeta-z_Phi_int_1}, we see that 
		\begin{align*}
		\left| \int_D \frac{ D^{\mu_0} f_\ve(\zeta) X_2(z,\zeta)}{\Phi^{n+2} (z,\zeta)} \, dV(\zeta) \right| 
		&\lesssim |\rho|_{k+3} |f_\ve|_k \int_D \frac{|\zeta-z|}{|\Phi (z,\zeta)|^{n+2}} \, dV(\zeta) 
		\\ &\lesssim |\rho|_{k+3} |f_\ve|_k \int_D \frac{ dV(\zeta)}{|\zeta-z| |\Phi (z,\zeta)|^{n+1}} 
		\lesssim  |\rho|_{k+3} |f_\ve|_k \left( 1 + \del(z)^{-\yh} \right), \quad \mu_0 \leq k. 
		\end{align*}  
		On the other hand, we can write 
\begin{equation} \label{Lopt_X1_int} 
\begin{aligned} 
  &\int_D \frac{ D^{\mu_0} f_\ve(\zeta) X_1(z,\zeta)}{\Phi^{n+2} (z,\zeta)} \, dV(\zeta) 
\\ &\quad = \int_D \frac{ [D^{\mu_0}_\zeta f_\ve(\zeta)- D^{\mu_0}_z f_\ve(z)] X_1(z,\zeta)}{\Phi^{n+2} (z,\zeta)} \, dV(\zeta) + D^{\mu_0}_z f_\ve(z) \int_D \frac{X_1(z,\zeta)}{\Phi^{n+2} (z,\zeta)} \, dV(\zeta).  
\end{aligned}	
\end{equation}
Since $f_\ve \in C^{k+\beta}(\ov D)$, the first integral on the right-hand side above is bounded up to a constant by 
\begin{align*}
  | \rho|_{k+2} |f_\ve|_{k+\beta} \int_D \frac{|\zeta-z|^\beta}{|\Phi(z,\zeta)|^{n+2}} \, dV(\zeta) 
&\lesssim    | \rho|_{k+2} |f_\ve|_{k+\beta} \int_D \frac{dV(\zeta)}{|\zeta-z|^{2-\beta} |\Phi(z,\zeta)|^{n+1}} 
\\ & \lesssim |\rho|_{k+2} |f_\ve|_{k+\beta} \left(1+\del(z)^{-1+\frac{\beta}{2}} \right).    
\end{align*}
For the second integral on the right-hand side of \re{Lopt_X1_int},  we can integrate by parts and bound the resulting expression by
		\[
		| \rho|_{k+3} |f_\ve|_k  \left( \int_{bD} \frac{d \si( \zeta)}{|\Phi(z,\zeta)|^n} + \int_D \frac{dV(\zeta)}{|\Phi(z,\zeta)|^{n+1} }\right)  \lesssim |\rho|_{k+3} |f_\ve|_k (1+ \log  \del(z)). 
		\]
Hence we have shown that 
 \[
  |I_2| \lesssim |f_\ve|_{k+\beta} |\rho|_{k+3} 
\left( 1+ \del(z)^{-1 + \frac{\beta}{2}} \right). 
 \]
Combined with the estimate \re{Lf_I1_est} for $I_1$ , this shows that $|D_z^{k+1} \Lc f_\ve (z) | \lesssim |f_\ve|_{k+\beta} |\rho|_{k+3} \left( 1+ \del(z)^{-1 + \frac{\beta}{2}} \right)$. By \rl{Lem::H-L}, $\Lc f_\ve \in C^{k+\frac{\beta}{2}}(\ov D)$ and $ |\Lc f_\ve|_{k+\frac{\beta}{2}} \leq C_0' |f_\ve|_{k+\beta}$ where $C_0'$ depends only on $|\rho|_{k+3}$. This proves the claim and hence the case when $0 < \beta <1$. 
Finally if $\beta =1$, the same proof works without the use of the approximation.  
		\\ \\
		\medskip 
		(ii)
		From \rp{Prop::K_bdd_intro} we know that $\Kc f \in C^{k+ \min\{ \all,  \yh\}}(\ov D)$ if $f \in C^k(\ov D)$ (and in particular if $f \in C^{k+\beta}(\ov D)$ for $0 < \beta \leq 1$). By (i), $\Lc f \in C^{k+\frac{\beta}{2}} (\ov D)$.
  Since $ \Lc^\ast f = \Kc f + \Lc f $, and $\min\{ \all, \yh, \frac{\beta}{2} \} = \min \{ \all, \frac{\beta}{2} \}$, we have 
\[
  \Lc^\ast f \in C^{ k+ \min \{ \all, \frac{\beta}{2} \}}(\ov D) . 
\]
 Finally by the integral equation $(I+\Kc) Pf =  \Lc^\ast f$, and the fact that $I +\Kc$ is invertible in the space $C^{k}(\ov D)$, we get $P f \in C^k(\ov D)$ and thus $\Kc Pf \in C^{k+ \min \{\all, \yh\}}(\ov D)$ by \rp{Prop::K_bdd_intro}. Therefore $Pf = -\Kc Pf + \Lc^\ast f \in C^{k+ \min \{ \all, \frac{\beta}{2} \}}(\ov D)$.  
	\end{proof}

	\bibliographystyle{amsalpha}
	\bibliography{Reference_Bergman} 
	
\end{document}